\documentclass[12pt]{amsart}

\usepackage[english]{babel} 
\usepackage[utf8]{inputenc}
\usepackage{cmbright}
\usepackage[T1]{fontenc}
\usepackage{lmodern} 
 \usepackage{array,multirow}

\usepackage{float} 
\usepackage{comment}
\usepackage{cite} 

\usepackage{caption}%
\usepackage[math]{cellspace}
\def\pwidth{48mm}

\usepackage{amsmath,amssymb,amsthm,mathtools,color}
\usepackage{fancyhdr}
\usepackage{geometry} 

\usepackage{xcolor}
\usepackage{hyperref}
\definecolor{darkgreen}{rgb}{0,0.4,0}
\definecolor{BrickRed}{rgb}{0.65,0.08,0}
\hypersetup{colorlinks=true,linkcolor=blue,citecolor=red,filecolor=BrickRed,urlcolor=darkgreen}

\usepackage{graphicx} 
\graphicspath{{pics/}}

\usepackage{maplestd2e} 
\mathtoolsset{showonlyrefs} 

\numberwithin{equation}{section}

\newtheorem{theorem}{Theorem}[section]
\newtheorem{lemma}[theorem]{Lemma}
\newtheorem{proposition}[theorem]{Proposition}

\newtheorem{conjecture}[theorem]{Conjecture}

\theoremstyle{example}
\newtheorem{example}[theorem]{Example}

\theoremstyle{definition} 
\newtheorem{definition}[theorem]{Definition}
\theoremstyle{remark}
\newtheorem*{remark}{Remark} 

\newcommand{\walksym}{\omega}
\newcommand{\walk}[1]{\walksym_{#1}}
\def\N{\mathbb{N}}
\def\Z{\mathbb{Z}}

\def\S{\mathcal{S}}
\def\A{\mathcal{A}}
\def\T{{\mathcal T}}
\def\boxproduct{\hspace{1mm}{}^\square\hspace{-1mm}\times}
\newcommand{\oeis}[1]{\text{\href{https://oeis.org/#1}{{\small \tt #1}}}} 
\newcommand{\OEIS}[1]{\text{\href{https://oeis.org/#1}{{\small \tt (#1)}}}} 

\begin{document}
\title[Formulas for enumeration of lattice paths: the kernel method] 
{Explicit formulas for enumeration of lattice paths: basketball and the kernel method} 



\author[Banderier, Krattenthaler, Krinik, Kruchinin, Kruchinin, Nguyen, Wallner]{Cyril Banderier \and Christian Krattenthaler \and Alan Krinik \and Dmitry Kruchinin \and Vladimir Kruchinin \and David 
Nguyen \and Michael Wallner}

\dedicatory{
\begin{tabular}{cc}
\begin{tabular}{c}\includegraphics[height=3cm]{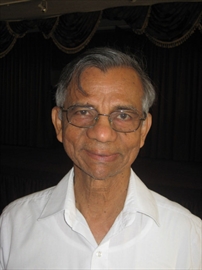}\end{tabular} & \begin{tabular}{l} Dedicated to Sri Gopal Mohanty,\\ a pioneer in the field of lattice paths combinatorics,\\ on the occasion of his 84th birthday.\end{tabular}
\end{tabular}}

\newcount\m \newcount\n
\def\hours{\n=\time \divide\n 60
	\m=-\n \multiply\m 60 \advance\m \time
	\twodigits\n:\twodigits\m}
\def\twodigits#1{\ifnum #1<10 0\fi \number#1}
\date{\today}

\address{\rm Cyril Banderier, CNRS/Univ. Paris 13, Villetaneuse (France).\protect\newline
	{\href{http://lipn.fr/~banderier/}{\tt http://lipn.fr/$\sim$banderier}}
}

\address{\leavevmode\vskip-5mm
\rm Christian Krattenthaler, Fakult\"{a}t f\"{u}r Mathematik, Universit\"{a}t Wien (Austria). \protect\newline
{\href{http://www.mat.univie.ac.at/~kratt/}{\tt http://www.mat.univie.ac.at/$\sim$kratt}}
}

\address{\leavevmode\vskip-5mm
\rm Alan Krinik, Department of Mathematics and Statistics, Cal Poly Pomona (USA). 
}

\address{\leavevmode\vskip-5mm
\rm Dmitry Kruchinin, Tomsk State University of Control Systems and Radio Electronics (Russia). 
}

\address{\leavevmode\vskip-5mm
\rm Vladimir Kruchinin, Tomsk State University of Control Systems and Radio Electronics (Russia). 
}

\address{\leavevmode\vskip-5mm
\rm David Nguyen, Department of Mathematics, University of California, Santa Barbara (USA). 
\protect\newline
	{\href{http://www.math.ucsb.edu/~dnguyen/}{\tt http://www.math.ucsb.edu/$\sim$dnguyen}}
}

\address{\leavevmode\vskip-5mm
\rm Michael Wallner, Institute of Discrete Mathematics and Geometry, TU Wien (Austria). \protect\newline
	{\href{http://dmg.tuwien.ac.at/mwallner/}{\tt http://dmg.tuwien.ac.at/mwallner}}
}

\begin{abstract}
This article deals with the enumeration of 
directed lattice walks on the integers with any finite set of steps,
starting at a given altitude $j$ and ending at a given altitude $k$,
with additional constraints such as, for example, to never attain
altitude $0$ in-between.
We first discuss the case of walks on the integers with steps 
$-h, \dots, -1, +1, \dots, +h$.
The case $h=1$ is equivalent to the classical Dyck paths, 
for which many ways of getting explicit formulas involving 
Catalan-like numbers are known. The case $h=2$ corresponds 
to ``basketball'' walks, which we treat in full detail.
Then we move on to the more general case of walks with any finite set 
of steps, also allowing some weights/probabilities associated with each step. 
We show how a method of wide applicability, the so-called ``kernel method'',
leads to explicit formulas for the number of walks of length $n$, 
for any $h$, in terms of nested sums of binomials. 
We finally relate some special cases to other combinatorial problems, 
or to problems arising in queuing theory.
\end{abstract}
\maketitle  

\pagebreak
\tableofcontents

\vspace{-13mm}
\section{Introduction}\label{Section1}

While analysing permutations sortable by a stack, 
Knuth~\cite[Ex.~1--4 in Sec.~2.2.1]{knuth} 
showed they were counted by Catalan numbers,
and were therefore in bijection with Dyck paths (lattice paths with
steps $(1,1)$ and $(1,-1)$ in the plane integer lattice, from the
origin to some point on the $x$-axis, and never running below the
$x$-axis in-between). He used 
a method to derive the corresponding generating function
(see \cite[p.~536ff]{knuth}) which Flajolet coined
{\it``kernel method''}. That name stuck among combinatorialists, although
the method already existed in the folklore of statistics and statistical physics 
--- without a name.  
The method was later generalized to enumeration and asymptotic analysis of
directed lattice paths with any set of steps, 
and many other combinatorial structures enumerated by bivariate or trivariate functional equations
(see, e.g., \cite{Fayolle99, bope00, BaBoDeGaGo02, BaFl02, BoJe06, Eynard16}).
We refer to the introduction of~\cite{BaWa16} for a more detailed history of the kernel method.

The emphasis in \cite{BaFl02} is on asymptotic analysis, for which the
derived (exact) enumeration results serve as a starting point. The latter are in a
sense {\it implicit}, since they involve solutions to certain
algebraic equations. They are nevertheless perfect for carrying out
singularity analysis, which in the end leads to very precise 
asymptotic results. 

In general, it is not possible to simplify the exact enumeration results 
from~\cite{BaFl02}. 
However, for models involving special choices of step sets, this is possible.
These potential simplifications are the main focus of our paper.

\begin{figure}[hb]
{\includegraphics[width=109mm]{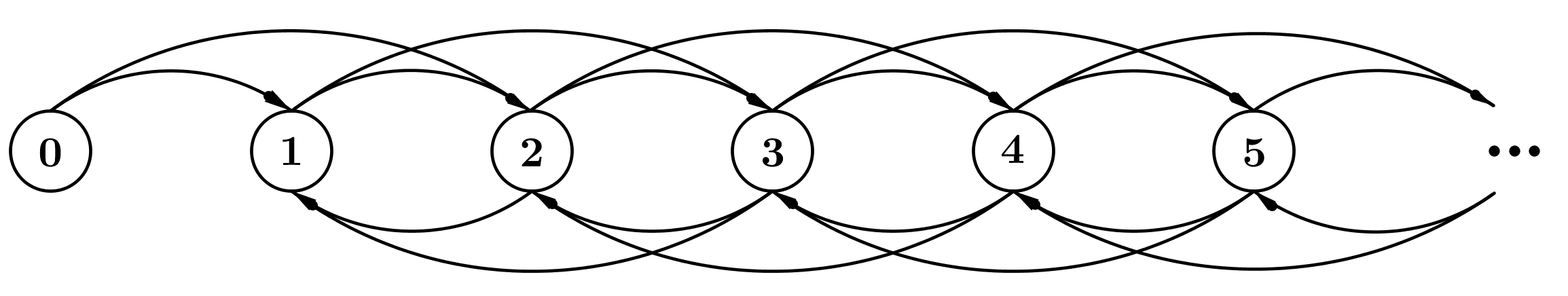}} 
\caption{A queue corresponding to the basketball walk model.}
\label{queue}
\end{figure}

Such models appear frequently in queuing theory. 
Indeed, birth and death processes and queues, like the one 
shown in Figure~\ref{queue}, are naturally encoded by lattice paths 
(see~\cite{Bohm10,Harris08,KrMo10,KrinikShun11,Margolius17,Mohanty07}).  
In this article, we solve a problem raised during the 2015 International Conference 
on ``Lattice Path Combinatorics and Its Applications'':
to find closed-form formulas for the number of walks of length $n$
from~$0$ to $k$ for a full family of models similar to
Figure~\ref{queue}. 
As it turns out, the essential tool to achieve this goal is indeed the kernel method.

\medskip
Our paper is organized as follows.
We begin with some preliminaries in Section~\ref{Section2}.
In particular, we introduce the directed lattice paths that we are
going to discuss here, we provide a first glimpse of the kernel
method, and we briefly review the Lagrange--B\"urmann inversion
formula for the computation of the coefficients of implicitly defined
power series. Section~\ref{Section3} is devoted to (old-time)
``basketball walks", which, by definition, are directed lattice walks 
with steps from the set $\{(1,-2),(1,-1),(1,1),(1,2)\}$ which always
stay above the $x$-axis.
(They may be seen as the evolution of --- pre~1984 --- basketball
games; see the beginning of that section for a more detailed 
explanation of the
terminology.) We provide exact formulas (often several, not obviously
equivalent ones) for generating functions
and for the numbers of walks under various constraints.
At the end of Section~\ref{Section3}, we also briefly address the
asymptotic analysis of the number of these
walks. Section~\ref{Section4} then considers the more general problem
of enumerating directed walks where the allowed steps are of the
form $(1,i)$ with $-h\le i\le h$ (including $i=0$ or not).
Again, we provide exact formulas for generating functions ---
in terms of roots of the so-called kernel equation --- and for numbers
of walks --- in terms of nested sums of binomials. All these results
are obtained by appropriate combinations of the kernel method with
variants of the Lagrange--B\"urmann inversion formula. 
In the concluding Section~\ref{Section5},
we relate basketball walks with other combinatorial objects, namely
\begin{itemize}
\item with certain trees coming from option pricing,
\item with increasing unary-binary trees which avoid a certain pattern
which arose in work of Riehl~\cite{Riehl16},
\item and with certain Boolean bracketings which appeared in work of
Bender and Williamson~\cite{bender91}.
\end{itemize}


\section{The general setup, and some preliminaries}\label{Section2}

In this section, we describe the general setup that we consider in
this article. 
We use (subclasses of) so-called {\it \L ukasiewicz paths}
as main example(s) which serve to illustrate this setup.
We recall here as well the main tools that we shall use in this
article: the {\it kernel method\/} and 
the {\it Lagrange--B\"urmann inversion formula}.

We start with the definition of the lattice paths under consideration.
 \begin{definition} \label{def:LP}
 A {\it step set} $\S \subset \Z^2$ is a finite set of vectors 
$$\{ (x_1,y_1), (x_2,y_2), \ldots, (x_m,y_m)\}.$$
An $n$-step \emph{lattice path} or \emph{walk} is a sequence of vectors $v = (v_1,v_2,\ldots,v_n)$, such that $v_j$ is in $\S$. 
Geometrically, it may be interpreted as 
a sequence of points $\walksym =(\walk{0},\walk{1},\ldots,\walk{n})$, where $\walk{i} \in \Z^2, \walk{0} = (0,0)$ (or another starting point),
and $\walk{i}-\walk{i-1} = v_i$ for $i=1,2,\ldots,n$.
The elements of $\S$ are called \emph{steps}. 
The \emph{length} $|\walksym|$ of a lattice path is its number $n$ of steps. 
 \end{definition}
The lattice paths can have different additional constraints shown in Table~\ref{fig:4types}.
\vspace{-1mm}
\begin{table*}[ht]
 \small
 \begin{center}\renewcommand{\tabcolsep}{3pt}
 \begin{tabular}{|c|c|c|}
 \hline
 & ending anywhere & ending at 0\\
 \hline  
 \begin{tabular}{c} unconstrained \\ (on~$\Z$) \end{tabular}
 & \begin{tabular}{c} 
 {\includegraphics[width=\pwidth]{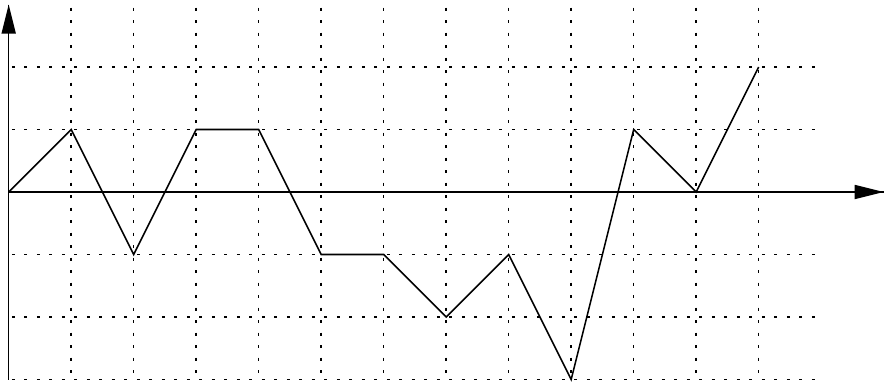}} 
\\ 
 walk/path ($\mathcal W$) 
 \end{tabular}
 & \begin{tabular}{c} 

 {\includegraphics[width=\pwidth]{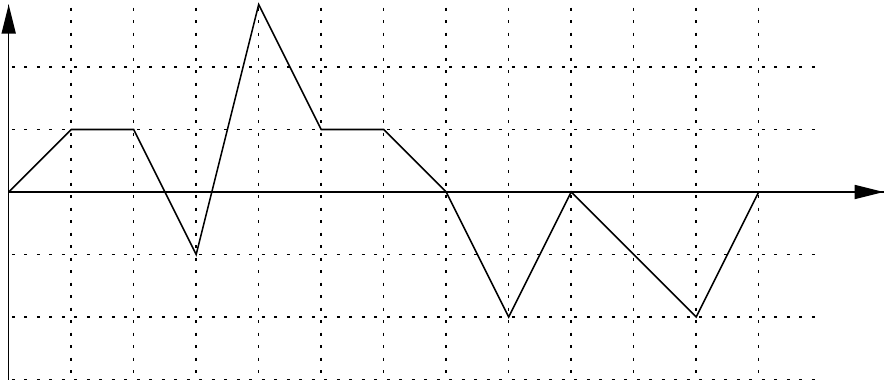}} 
\\
 bridge ($\mathcal B$)
 \end{tabular} \\
 \hline
 \begin{tabular}{c}constrained\\ (on $\N$) \end{tabular}
 & \begin{tabular}{c} 
 \includegraphics[width=\pwidth]{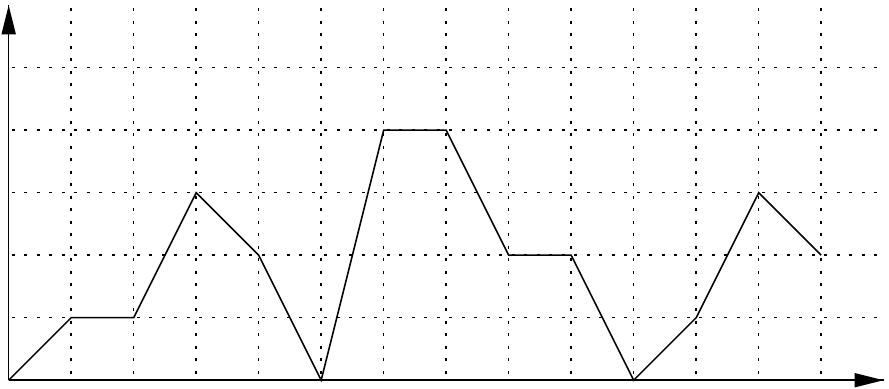} 
\\ 
 meander ($\mathcal M$)\\ 
 \end{tabular}
 & \begin{tabular}{c} 
 {\includegraphics[width=\pwidth]{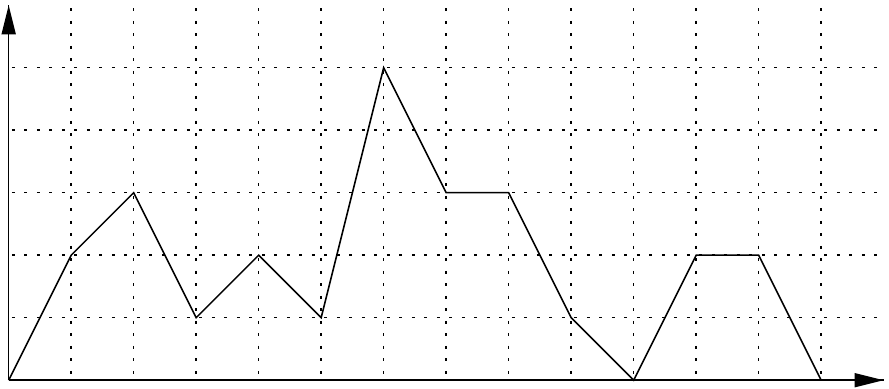}} 
\\ 
 excursion ($\mathcal E$)\\ 
 \end{tabular}\\
 \hline
 \end{tabular}
 \end{center}
\vskip2pt
 \caption{\label{fig:4types} 
The four types of walks: unconstrained walks, bridges, meanders, 
and excursions.}
 \end{table*}

\vspace{-7mm}
We restrict our attention to \emph{directed paths}, which are defined by the fact that, for each step $(x,y) \in \S$, one has $x \geq 0$. 
Moreover, we will focus only on the subclass of \emph{simple paths}, where every element in the step set $\S$ is of the form $(1,b)$. 
In other words, these paths constantly move one step to the right.
Thus, they are essentially one-dimensional objects and can be seen as
walks on the integers. 
We introduce the abbreviation $\S = \{ b_1,b_2, \ldots, b_n \}$ in this case. 
A \emph{{\L}ukasiewicz path} is a simple path where its associated step set $\S$ is a subset of $\{-1,0,1,\ldots\}$ and $-1 \in \S$.

\begin{example}[\sc Dyck paths]
	\label{ex:dyck}
	A Dyck path is a path constructed from the step set $\S = \{-1,+1\}$, which starts at the origin, never passes below the $x$-axis, and ends on the $x$-axis. In other words, Dyck paths are excursions 
with step set $\S=\{-1,+1\}$.
\end{example}

The next definition allows to merge the probabilistic point of view ({\it random walks}) and the combinatorial point of view ({\it lattice paths}).

 \begin{definition}
 For a given step set $\S = \{s_1,s_2,\ldots,s_m\}$, we define the corresponding {\it system of weights}\index{lattice path!weights} 
as $ \{p_1,p_2,\ldots,p_m\}$, where $p_j >0$ is the weight associated with step $s_j$ for $j=1,2,\ldots,m$. 
The {\it weight of a path} is defined as the product of the weights of its individual steps. 
 \end{definition}


Next we introduce the algebraic structures associated with the previous definitions.
The \emph{step polynomial\/} of a given step set $\S$ is defined as the Laurent polynomial\footnote{By a {\it Laurent polynomial\/} in $u$ we mean a polynomial in $u$ and $u^{-1}$.}
\begin{align*}
	P(u) := \sum_{j=1}^m p_j u^{s_j}.
\end{align*}
Let 
\begin{equation}\label{c_d}
c = - \min_j s_j\quad  \text{and}\quad  d = \max_j s_j
\end{equation} 
be the two extreme step
sizes, and assume throughout that $c,d >0$. 
Note that for {\L}ukasiewicz paths we have $c=1$. 

We start with the easy case of unconstrained paths.
We define their bivariate generating function as 
$$W(z,u) = \sum_{n=0}^\infty \sum_{k=-\infty}^\infty W_{n,k} z^n u^k,$$ 
where $W_{n,k}$ is the number of unconstrained paths ending after $n$
steps at altitude $k$.  

It is well-known and straightforward to derive that
\begin{align}
	\label{eq:unconstrainedpaths}
	W(z,u) &= \frac{1}{1- zP(u)}.
\end{align}

We continue with the generating function of meanders:
\begin{align*}
	F(z,u):=\sum_{n=0}^\infty \sum_{k=0}^\infty F_{n,k}z^nu^k,
\end{align*}
where $F_{n,k}$ is the number of paths ending after $n$ steps at
altitude $k$, and constrained to be always at altitude $\geq 0$
in-between. 
Note that we are mainly interested in solving the counting problem,
i.e., determining the numbers $F_{n,k}$ for specific families of paths
(see Table~\ref{fig:4types}). The generating function encodes all
information we are interested in.  

We decompose $F(z,u)$ in two ways, namely
\begin{align*}
	F(z,u) = \sum_{k \geq 0} F_k(z) u^k = \sum_{n \geq 0} f_n(u) z^n.
\end{align*}
Here, the generating functions $F_k(z)$ enumerate paths ending at
altitude $k$, i.e., $F_k(z) = \sum_{n \geq 0} F_{n,k} z^n$. 
In particular, the
generating function for excursions is equal to $F_0(z)$. On the other
hand, the polynomials $f_n(u)$ enumerate paths of length $n$. The
power of $u$ encodes their final altitude.  
We will use this decomposition for a step-by-step approach, similar to
the one in the case of unconstrained paths. 

For the sake of illustration, we show below how the kernel method can
be used to find a closed form for the generating function of a
given class of {\L}ukasiewicz paths.
\begin{theorem}
	\label{theo:luka}
	Let $\S$ be the step set of a class of {\L}ukasiewicz paths, 
and let $P(u)$ be the associated step polynomial. Then the bivariate generating function of meanders (where $z$ marks length, and $u$ marks final altitude) and excursions are 
	\begin{align} \label{eqFzuLuka}
		F(z,u) &= \frac{1 - zF_0(z)/u}{1-zP(u)} & \text{and} &&
		F_0(z) &= \frac{u_1(z)}{z}, 
	\end{align}
respectively, where $u_1(z)$ is the unique small solution of the implicit equation
	\begin{align*}
		1 - zP(u) =0,
	\end{align*}
that is, the unique solution
	satisfying $\lim_{z \to 0} u_1(z) = 0$.
\end{theorem}
\begin{proof}
	A meander of length $n$ is either empty, or it is constructed
        from a meander of length $n-1$ by appending a possible step
        from $\S$. However, a meander is not allowed to pass below the
        $x$-axis, thus at altitude $0$ it is not allowed to use the
        step $-1$. This translates into the relations
	\begin{align*}
		f_0(u)  &= 1, &
		f_{n+1}(u) &= \{u^{\geq 0}\} \left( P(u) f_n(u) \right),
	\end{align*}
	where $\{u^{\geq 0}\}$ is the linear operator extracting all terms in the power series representation containing non-negative powers of $u$. Multiplying both sides of the above equation by $z^{n+1}$ and subsequently 
summing over all $n \geq 0$, we obtain the functional equation 
	\begin{align*}
		F(z,u) &= 1+ zP(u) F(z,u) - \frac{z}{u} F_0(z).
	\end{align*}
Equivalently,
	\begin{align}
		\label{eq:genFuncEq}
{\left(1-zP(u)\right)} F(z,u) &= 1 - \frac{z}{u} F_0(z)\,.
	\end{align} 
We write $K(z,u):=1-zP(u)$ and call this
factor $K(z,u)$ the \emph{kernel\/}. The above functional equation
looks like an underdetermined equation as there are two unknown functions, namely $F(z,u)$ and $F_0(z)$. However, the special structure on the left-hand side will resolve this problem and leads us to the \emph{kernel method}.
\smallskip

Using the theory of Newton polygons and Puiseux expansions
(cf.~\cite[Appendix of Sec.~3]{DieuIC}), 
	we know that the \emph{kernel equation}
		\begin{align}
			\label{eq:kerneleq}
			1-zP(u)=0,
		\end{align}
		has $d+1$ distinct solutions in $u$ (recall that $c=1$, see Equation~\eqref{c_d}).
One of them, say $u_1(z)$, maps $0$ to $0$. We call this solution
the ``small branch'' of the kernel equation. 
It is in modulus smaller than the other $d$ branches.
These in turn grow to infinity in modulus while $z$ approaches~0.
Consequently, we call the latter the ``large branches'' and denote them by
$v_1(z),v_2(z),\dots,v_d(z)$.
		Inserting the small branch into~\eqref{eq:genFuncEq} (this is legitimate as we stay in 
		the integral domain of Puiseux power series: substitution of the small branch always leads to series having a finite number of terms with negative exponents, even for intermediate computations), we get 
		$F_0(z) = u_1(z)/z$.		
This proves our second claim.
	Using this result, we can solve~\eqref{eq:genFuncEq} for $F(z,u)$ to get the first claim.	
\end{proof}

The formula~\eqref{eqFzuLuka} in the previous theorem implies that the number $m_n$ of meanders of length $n$ 
is directly related to the number $e_n$ of excursions of length $n$ via 
\begin{equation}m_n = P(1)^n - \sum_{k=0}^{n-1}  P(1)^{k} e_{n-k-1}.\label{meanderformula}
\end{equation}
In the sequel, we therefore focus on giving explicit expressions for $e_n$.

\medskip
A key tool for finding a formula for the coefficients of power series satisfying 
implicit equations is the 
Lagrange inversion formula \cite{Lagrange70}, independently discovered
in a slightly extended form by B\"urmann \cite{Buermann} (see also
\cite{LaLeAA}).
In the statement of the theorem and also later, we use the
 \emph{coefficient extractor} $[z^n] F(z) :=f_n$ for a power series $F(z)=\sum f_n z^n $.

\begin{theorem}[\sc Lagrange--B\"urmann inversion formula]
\label{thm:LagrangeBurmann}
Let $F(z)$ be a formal power series which satisfies $F(z)=z \phi(F(z))$, 
where $\phi(z)$ is a power series with $\phi(0)\neq 0$. 
Then, for any Laurent\footnote{Here, by Laurent series we mean a series of the form $H(z)=\sum_{n\ge a} H_n\,z^n$ for some (possibly negative) integer~$a$.} series $H(z)$ and for all non-zero integers $n$, we have
$$[z^n] H(F(z)) = \frac{1}{n} [z^{n-1}] H'(z) \phi^n(z) \,.$$
\end{theorem}
\begin{proof}
See~\cite[Chapter~A.6]{FlSe09} or~\cite[Theorem~5.4.2]{stan99}.
\end{proof}

\pagebreak
\begin{table}[ht]\scalebox{0.94}{\hspace{-2mm}
\begin{tabular}{|Sc|Sc|Sc|}
\hline
\begin{tabular}{c} name and the associated \\ 
step polynomial $P(u)$\end{tabular} & number $e_n$  of excursions of length $n$ \\
\hline
\begin{tabular}{c}
Dyck paths \\
$P(u)=\frac{1}{u}+u$ \end{tabular} & $\displaystyle{e_{2n}= \frac{1}{n+1} \dbinom{2n}{n}}$\\
\hline
\begin{tabular}{c}Motzkin paths \\ 
$P(u)=\frac{1}{u}+1+u$\end{tabular}  & $\displaystyle{e_n=\frac{1}{n+1} \sum_{k=0}^{\lceil \frac{n+1}{2} \rceil} \dbinom{n+1}{k} \dbinom{n+1-k}{k-1}}$\\
\hline
\begin{tabular}{c}weighted Motzkin paths \\ $P(u)=\frac{p_{-1}}{u}+p_0+p_1 u$ \end{tabular} & $\displaystyle{e_{n}=\frac{1}{n+1} \sum_{k=0}^{\lceil \frac{n+1}{2} \rceil} \dbinom{n+1}{k} \dbinom{n+1-k}{k-1} (p_1 p_{-1})^{k-1} p_0^{n+2-2k}}$\\
\hline
\begin{tabular}{c}bicoloured Motzkin paths\\ $ P(u)=\frac{1}{u}+2+u $\end{tabular}  & $\displaystyle{e_{n+1}=\frac{1}{n+1}\dbinom{2n}{n}}$\\
\hline
\begin{tabular}{c}{\L}ukasiewicz paths\\
$P(u)=\frac{1}{u}+1+u+u^2+\cdots$
\end{tabular}
& $\displaystyle{e_{n}=\frac{1}{n+1}\dbinom{2n}{n}}$\\
\hline
\begin{tabular}{c}$d$-ary trees \\$P(u)=\frac{1}{u}+u^{d-1}$\end{tabular}& $e_{dn+1}=\displaystyle{\frac{1}{(d-1)n+1}\dbinom{dn}{n}}$\\
\hline
\begin{tabular}{c} $\{1,2,\dots,d\}$-ary trees \\$P(u)=\frac{1}{u}+1+\dots+u^{d-1}$\end{tabular}& $e_{n}=\displaystyle{ \frac{1}{n}\sum_{j=0}^{\lfloor{\frac{n-1}{d+1}}\rfloor}   (-1)^j \dbinom{n}{ j}\dbinom{2n-2-j(d+1)}{n-1}
\ }$\\
\hline
\begin{tabular}{c}$\{d,d+1\}$-ary trees\\$P(u)=\frac{1}{u}+u^{d-1}+u^d$\end{tabular}& $\displaystyle{e_n=\frac{1}{n}\sum_{k=0}^{\lfloor \frac{n-1}{d} \rfloor} \dbinom{n}{k} \dbinom{k}{n-1-dk}}$\\
\hline
\end{tabular}
}
\vskip8pt
\caption{Closed-form formulas for some famous families of lattice
  paths. 
}
\label{TableLukaPaths}
\end{table}

\vspace{-5mm}
Table~\ref{TableLukaPaths} presents several applications of this Lagrange inversion 
formula to lattice path enumeration.
It leads to the Catalan numbers for Dyck paths, 
and to the Motzkin numbers for the Motzkin paths, i.e., excursions associated with the step set $\S = \{-1,0,+1\}$.
They are two of the most  ubiquitous number sequences in combinatorics, 
see~\cite[Ex.~6.19,  6.25, and 6.38]{stan99} for more information.
Table~\ref{TableLukaPaths} also contains an example of weighted paths
(namely weighted Motzkin paths and the special case of bicoloured
Motzkin paths), as well as an example with an infinite set of steps
(namely the \L ukasiewicz paths with all possible steps allowed).

All of the examples in Table~\ref{TableLukaPaths} are intimately
related to families of trees (as suggested by some of the namings in
the table).
In order to explain this, we recall that
an \emph{ordered tree} is a rooted 
tree for which an ordering of the children is specified for each vertex,
and for which its arity (i.e., the outdegree, the number of children
of each node) is restricted to be in a subset $\A$ of
$\N$.\footnote{In this article, by convention $0\in \N$.}
If $\A=\{0,2\}$, this leads to the classical binary trees counted by the Catalan numbers;
if $\A=\{0,1,2\}$, this leads to the unary-binary trees counted by Motzkin numbers, and if $\A=\N$, this gives the ordered trees (also called planted plane trees), which are also counted by Catalan numbers.
Any ordered tree can be traversed starting from the root in 
\emph{prefix order}: 
one starts from the root and proceeds depth-first and left-to-right. 
The listing of the outdegrees of nodes in prefix order is called the \emph{preorder degree sequence}. This 
characterizes a tree unambiguously, see Figure~\ref{lukabijection}, 
and it is best summarized by the following folklore proposition.

\begin{proposition}[\sc {\L}ukasiewicz correspondence]
	Ordered trees 
	are in bijection with {\L}ukasiewicz excursions.
\end{proposition}

\begin{proof}
Given an ordered tree with $n$ nodes,
the preorder sequence can be interpreted as a lattice path. 
Let $(\sigma_j)_{j=1}^{n}$ be a preorder degree sequence. With each $\sigma_j$ we associate a step $(1,\sigma_j-1) \in \N\times\Z$. 
Note that, as the minimal degree is $0$, our smallest step is $-1$. 
Starting at the origin, we concatenate these steps for $j=1,2,\dots,n-1$,
ignoring the last step. In this way, we obtain
a {\L}ukasiewicz excursion of length $n-1$.
\end{proof}

\begin{figure}[!ht]
\scalebox{0.88}{
	\includegraphics[height=25mm]{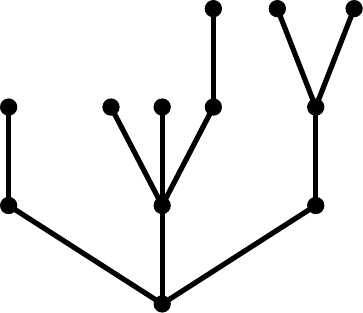} 
	\qquad \qquad
 \includegraphics[height=25mm]{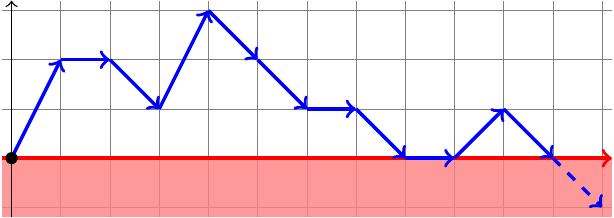} 
 }
	\caption{The bijection between trees and {\L}ukasiewicz
          paths. The preorder degree sequence
          $(3,1,0,3,0,0,1,0,1,2,0,0)$ uniquely characterizes the tree,
          and gives the corresponding {\L}ukasiewicz path with step sequence $(2,0,-1,2,-1,-1,-1,0,-1,0,1,-1,-1)$. Dropping the last $-1$ step yields an excursion.
	\label{lukabijection}}
\end{figure}


As one can see,
the combinatorics of the {\L}ukasiewicz paths is well understood  (see e.g.~\cite{FlSe09,Stanley86}),
and the true challenge is to 
analyse lattice paths with other negative steps than just $-1$.
The smallest non-{\L}ukasiewicz cases are the Duchon lattice paths (steps $\S=\{-2,+3\}$), and 
the Knuth lattice paths (steps $\S=\{-2,+5\}$). Their enumerative and asymptotic properties are the subject of another article in this volume~\cite{BaWa16}. 
For these two families of lattice paths, the asymptotics are tricky, because the generating functions involve several dominant singularities. 
In the next sections, we concentrate on closed formulas which appear for  many other non-{\L}ukasiewicz cases.

\section{(Old-time) Basketball walks: steps \texorpdfstring{$\S=\{-2,-1,+1,+2\}$}{S=\{-2,-1,+1,+2\}}}\label{Section3}

\begin{figure}[h!]\vspace{-3mm}
\includegraphics[height=45mm]{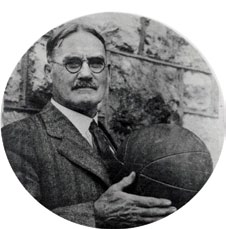}\vspace{-1mm}
\caption{Since its creation in 1892 by James Naismith (November 6, 1861 -- November 28, 1939), the rules of basketball evolved.
For example, since 1896, field goals and free throws were counted as two and one points, respectively.
The international rules were changed in 1984 so that a 
``far'' field goal was now rewarded by $3$ points, while ``ordinary'' field 
goals remained at $2$ points, a free throw still being worth one point.}
\label{fig:BB}
\end{figure}
We now turn our attention to a class of lattice paths (lattice walks) 
with rich combinatorial properties: \emph{the basketball walks}. 
They are constructed from the step set $\S=\{-2, -1, +1,+2\}$.
This terminology was introduced by Arvind Ayyer and Doron Zeilberger~\cite{ZeilbergerAyyer07}, and these walks were later also considered by Mireille Bousquet-M\'elou~\cite{BousquetMelou08}. They can be seen
as the evolution of the score during a(n old-time) basketball game
(see Figure~\ref{fig:BB}).

Ayyer, Zeilberger, and Bousquet-M\'elou 
found interesting results on the shape of the algebraic equations 
satisfied by the excursion generating function, 
and similar properties when the height of the excursion is bounded.
In this article, we analyse a generalization in which the starting point 
and the end point of the walks do not necessarily have altitude~$0$. 
Since, in that case, we lose a natural factorization happening for excursions,
we are led to variations of certain parts in the kernel method. 
In addition, we are interested in closed-form expressions for the number of walks of length $n$. 
This is complementary to the results in~\cite{BaFl02} and in~\cite{BaWa16}.
Moreover, contrary to the previous section, 
these walks are not {\L}ukasiewicz paths any more.
This makes them harder to analyse (the easy bijection with trees is
lost, for example).
Despite all that, the kernel method will strike again, thus illustrating our main motto:
$$
\text{\it ``The kernel method is the method of choice for problems on directed lattice paths!''}
$$

\subsection{Generating functions for positive (old-time) 
basketball walks: the kernel method}

We define \emph{positive walks} as walks staying strictly above the $x$-axis, possibly touching it at the first or last step. Returning to the basketball
interpretation, these correspond to the evolution of basketball scores
where one team (the stronger team, the richer team?) is always ahead of the other team.

Let $G_{j,n,k}$ be the number of such walks running from $(0,j)$ to $(n,k)$, and define by $G_j(z,u)$ the generating function of positive walks starting at $(0,j)$. We write
\begin{equation}
	G_j(z,u) := \sum_{n,k \geq 0}G_{j,n,k}z^nu^k = \sum_{n=0}^\infty g_{j,n}(u)z^n = \sum_{k=0}^\infty G_{j,k}(z)u^k.
\end{equation}
Similar to Section~\ref{Section2},
we shall need the polynomial $g_{j,n}(u)$, the
generating function for all walks with $n$ steps, and the series 
$G_{j,k}(z)$, the generating function for all walks ending 
at altitude $k$. 
The bivariate generating function $G_j(z,u)$ is analytic for
$|z|<1/P(1)$ and $|u|\le1$.  

A walk is either the single initial point at altitude $j$, or a walk
followed by a step not reaching altitude~0 or below. This leads to the
functional equation 
\begin{equation}
	(1-zP(u))G_j(z,u)=u^j-z\big(G_{j,1}(z)+G_{j,2}(z)+G_{j,1}(z)/u\big),
\quad \quad j>0, \label{fund.eq}
\end{equation}
 where the \textit{step polynomial\/} $P(u)$ is given by
\begin{equation*}
	P(u):=u^{-2}+u^{-1}+u+u^2.
\end{equation*}
Again, we call the factor $1-zP(u)$ on the left-hand side of
\eqref{fund.eq} the \textit{kernel\/} of the equation, and
denote it by $K(z,u)$.

We refer to~\eqref{fund.eq} as the \textit{fundamental functional
  equation} for $G_j(z,u)$. 
The equation has a small problem though: this is {\it one} equation with
{\it three} unknowns, namely $G_j(z,u), G_{j,1}(z)$, and $G_{j,2}(z)$! 
The idea of the so-called {\it`kernel method'} is to equate
the kernel $K(z,u)$ to $0$,
thus binding $u$ and $z$ in such a way that the left-hand side
of~\eqref{fund.eq} vanishes. This produces two extra equations.  

To equate $K(z,u)$ to zero means to put
\begin{equation} \label{eq:kernel.equation}
	1-zP(u)=0 \quad \text{ or equivalently } \quad u^2-zu^2P(u)=0.
\end{equation}
We call this equation the \textit{kernel equation}. 
As an equation of degree~4 in $u$, it has four roots. 
We call the two small roots (that is, the roots which tend to~0 when $z$
approaches~0) $u_1(z)$ and $u_2(z)$.

Then, on the complex plane slit along the negative real axis, 
we can identify the small roots $u_1(z)$ and $u_2(z)$ as
\begin{align*}
	u_1(z) &= -\frac{1}{4} \left( \frac{z - \sqrt{4z+9z^2}}{z} + \sqrt{\frac{4 - 6z - 2\sqrt{4z+9z^2}}{z}} \right)\\
	&=\sqrt{z}+\frac{1}{2}z+\frac{1}{8}z^{3/2}+\frac{1}{2}z^2+\frac{159}{128}z^{5/2}+O(z^3),\\
	u_2(z) &= -\frac{1}{4} \left( \frac{z + \sqrt{4z+9z^2}}{z} - \sqrt{\frac{4 - 6z + 2\sqrt{4z+9z^2}}{z}} \right)\\
	&=-\sqrt{z}+\frac{1}{2}z-\frac{1}{8}z^{3/2}+\frac{1}{2}z^2-\frac{159}{128}z^{5/2}+O(z^3).
\end{align*}
Moreover, their Puiseux expansions are related via the following proposition.
\begin{proposition}[\sc Conjugation principle for two small roots]\label{prop:conjugation}
The small roots $u_1$ and $u_2$ of\/ $1-zP(u)=0$ 
satisfy
\begin{equation}
u_1(z)=\sum_{n\geq 1} a_n z^{n/2} \text{ and } u_2(z) =\sum_{n\geq 1} (-1)^n a_n z^{n/2}\,.
\end{equation}
\end{proposition}
\begin{proof}
The kernel equation yields
$$
u=X(1+u+u^3+u^4)^{1/2},
$$
with $X=z^{1/2}$ or $X=-z^{1/2}$. Since the above equation possesses
a unique formal power series solution $u(X)$, the claim follows.
\end{proof}

By substituting the small roots $u_1(z)$ and $u_2(z)$ of the kernel equation
\eqref{eq:kernel.equation} into the fundamental functional equation \eqref{fund.eq}, we see that the left-hand side vanishes. Subsequently, we solve for $G_{j,1}(z)$ and $G_{j,2}(z)$ and get\footnote{In this article, whenever we thought it could ease the reading, without harming the understanding,
we write $u_1$ for $u_1(z)$, or $F$ for $F(z)$, etc.}
\begin{align}
	G_{j,1}(z)&=-\dfrac{u_1u_2(u_1^j-u_2^j)}{z(u_1-u_2)}, & j>0, \label{eq:G1}\\ 
	G_{j,2}(z)&=\dfrac{u_1u_2(u_1^j-u_2^j) + u_1^{j+1} - u_2^{j+1}}{z(u_1-u_2)}, & j>0. \label{eq:G2}
\intertext{Substitution in the fundamental functional equation \eqref{fund.eq} 
then yields}
	G_j(z,u) &= \frac{u^j-z(G_{j,1}(z)+G_{j,2}(z)+G_{j,1}(z)/u) }{1-zP(u)}, & j>0. \label{eq:G}
\end{align}

By means of the kernel method, we have thus derived an explicit expression for the bivariate generating function $G_j(z,u)$ for walks starting at altitude $j>0$. 

In the following proposition, we summarize our findings so far.
In addition, we express the generating function for walks from
altitude~$j$ to altitude~$k$ (with $j,k>0$)  explicitly in terms of the small
roots $u_1(z)$ and $u_2(z)$, and
we also cover the special case $j=0$, 
which offers some nice simplifications.

\begin{proposition}
	\label{prop:G0k}
As before,
	let $G_{j,k}(z)$ be the generating function for positive
basketball walks with steps $-2,-1,+1,+2$ starting at altitude $j$ and ending at altitude $k$. Furthermore, let $u_1(z)$ and $u_2(z)$ be the small roots of the kernel equation $1-zP(u)=0$, with $P(u)=u^{-2}+u^{-1}+u+u^2$. Then, for $j,k > 0$, we have
	\begin{align} 
	G_{0,k}(z) &= \frac{u_1^{k+1}(z)-u_2^{k+1}(z)}{u_1(z)-u_2(z)}
							, \label{eq:G0k} \\
	G_{j,k}(z) &= -\frac{u_1(z) u_2(z)}{z} 
			\sum_{i=0}^j \frac{u_1^{j-i+1}(z)-u_2^{j-i+1}(z)}{u_1(z)-u_2(z)} 
			    \frac{u_1^{k-i+1}(z)-u_2^{k-i+1}(z)}{u_1(z)-u_2(z)}, \label{eq:Gjkhi}
	\end{align}		
\end{proposition}

\begin{proof}
	We start with the proof of \eqref{eq:G0k}.
	The first step of a walk can only be a step of size $+1$ or $+2$. Thus, removing this first step and shifting the origin, we have 
	\begin{equation}
	G_{0,k}(z)=z \left( G_{1,k}(z)+G_{2,k}(z) \right),
	\end{equation}
	where $G_{1,k}(z)$ and $G_{2,k}(z)$ are the generating functions for positive walks running from altitude $1$ to altitude $k$, respectively
from altitude $2$ to altitude $k$. 
This decomposition is illustrated in Figure~\ref{G01walksinitialsteps}.
	
	\begin{figure}[!hb]
		\includegraphics[width=0.47\textwidth]{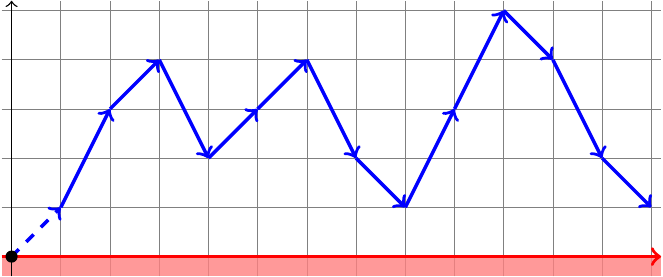} 
		\quad
		\includegraphics[width=0.47\textwidth]{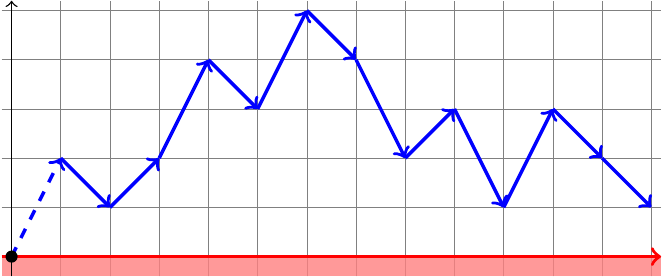} 
		\caption{Two different instances of walks counted by $G_{0,1}(z)$ showing the two possible first steps $+1$ and $+2$. 
		\label{G01walksinitialsteps}}
	\end{figure}
	
	By ``time reversal'' (due to the symmetry of our step set, i.e., $P(u)=P(u^{-1})$), we also have  
	\begin{equation}
	G_{1,k}(z)=G_{k,1}(z),\quad\text{ and }\quad G_{2,k}(z)=G_{k,2}(z),
	\end{equation} 
	where $G_{k,1}(z)$ and $G_{k,2}(z)$ are known from Equations~\eqref{eq:G1} and~\eqref{eq:G2}. Now notice that 
	\begin{align*}
	G_{k,2}(z) &= \dfrac{u_1u_2(u_1^k-u_2^k) + u_1^{k+1} - u_2^{k+1}}{z(u_1-u_2)}
	   =\dfrac{u_1u_2(u_1^k-u_2^k)}{z(u_1-u_2)}+\dfrac{u_1^{k+1} - u_2^{k+1}}{z(u_1-u_2)} \\
	   &=\frac{u_1^{k+1}-u_2^{k+1}}{z(u_1-u_2)}-G_{k,1}(z).
	\end{align*}
This leads directly to~\eqref{eq:G0k}.

\smallskip		
	For computing $G_{j,k}(z)$ with $j,k > 0$, we use a first passage decomposition with respect to minimal altitude of the walk.
Combining \eqref{eq:G0k} with time reversal, we see that
$h_m(z) := \frac{u_1^{m+1}-u_2^{m+1}}{u_1-u_2}$ is the generating
function for basketball walks starting at altitude $m$, staying always
        above the $x$-axis, but ending on the $x$-axis.
Furthermore, by \eqref{eq:G1} with $j=1$, the series
$E(z) = -\frac{u_1 u_2}{z}$ is the generating function for excursions 
(allowed to touch the $x$-axis).
Then the walks from altitude~$j$ to altitude~$k$ can be decomposed
into three sets, as illustrated by Figure~\ref{Gjkdecomp}:  
	\begin{enumerate}
		\item The walk starts at altitude $j$, and continues until it hits for the first time altitude~$i$ (the lowest altitude of the walk, so $1 \leq i \leq j$). This part is counted by $h_{j-i}(z)$. 
		\item The second part is the one from that point to the last time reaching altitude~$i$. In other words, this part is an excursion on level $i$ counted by $E(z)$.
		\item The last part runs from altitude $i$ to altitude
$j$ without ever returning to altitude~$i$. By time reversal one sees
                  that this is counted by $h_{k-i}(z)$. 
	\end{enumerate}	
	Summing over all possible $i$'s, we get~\eqref{eq:Gjkhi}.
\end{proof}

\begin{figure}[!hb]
 		\includegraphics[width=0.87\textwidth]{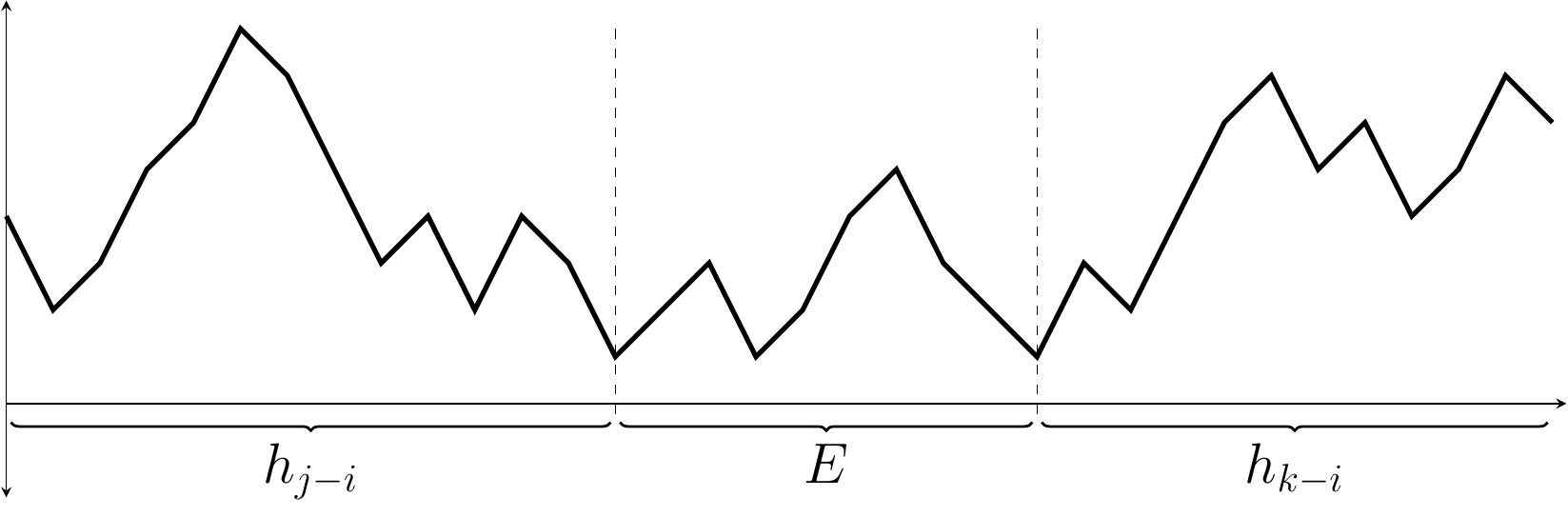} 
		\caption{The decomposition for $G_{j,k}$}
		\label{Gjkdecomp}
	\end{figure} 
	
There is an alternative expression for the generating function $G_{j,k}(z)$,
which we present in the next proposition.

\begin{proposition}[\sc Formula for walks from altitude $j$ to altitude $k$]
	\label{prop:Gjk}
Let $u_1(z)$ and $u_2(z)$ be the small roots of the kernel equation $1-zP(u)=0$, with $P(u)=u^{-2}+u^{-1}+u+u^2$,
	and let $G_{j,k}(z)$ be the generating function for positive
basketball walks starting at altitude $j$ and ending at altitude $k$. 
Then
	\begin{equation} \label{eq:uW}
	G_{j,k}(z)=W_{j-k} +
				 h_j(u_1,u_2) W_{-k} +
				 u_1 u_2 h_{j-1}(u_1,u_2) W_{-k+1},
	\end{equation}
	where $$W_i(z) = z \left( \frac{u_1'}{u_1^{i+1}} + \frac{u_2'}{u_2^{i+1}} \right)$$ is the generating function of unconstrained walks starting at the origin and ending at altitude $i$, and $$h_i(x_1, x_2) = \frac{x_1^{i+1}-x_2^{i+1}}{x_1-x_2}$$ is the complete homogeneous symmetric polynomial of degree $i$ in $x_1$ and $x_2$.
\end{proposition}

\begin{proof}
Since $G_{j,k}(z) = G_{k,j}(z)$, without loss of generality we may assume that $j\leq k$. We start with~\eqref{eq:G}. Extraction of the coefficient of $u^k$ on the left-hand side gives $G_{j,k}(z)$. As coefficient extraction is linear, we need to find expressions for
	\begin{align*}
		[u^i] \frac{1}{1-zP(u)}.
	\end{align*}
By~\eqref{eq:unconstrainedpaths}, these are the generating functions
$W_i(z)$ for unconstrained walks starting at the origin and ending at
altitude $i$. For basketball walks, we have $P(u)=P(u^{-1})$, hence
$W_i(z) = W_{-i}(z)$. 
Using a straightforward contour integral argument, using Cauchy's
integral formula and the residue theorem, we have
$$
W_i(z)=
[u^i] \frac{1}{1-zP(u)}=\frac {1} {2\pi \sqrt{-1}}\int_{\mathcal C}
\frac{du}{u^{i+1}(1-zP(u))}
=z \left( \frac{u_1'(z)}{u_1^{i+1}(z)} + \frac{u_2'(z)}{u_2^{i+1}(z)} \right).
$$
Thus, we obtain the claimed expression for $W_i(z)$ in terms of the
small branches. Finally, the remaining factors in \eqref{eq:uW} are obtained by
simplifications in~\eqref{eq:G}. 
\end{proof}

Thus, by~\eqref{eq:G0k}, walks starting at the origin are given by complete homogeneous symmetric polynomials in the small branches. In particular, we have 
\begin{align} 
	G_{0,1}(z) &=u_1(z)+u_2(z),\label{eq:GGGG01} \\
	G_{0,2}(z) &=u_1^2(z)+u_1(z)u_2(z)+u_2^2(z).\label{eq:GG02}
\end{align}

We now derive an explicit expression for $G_{0,1}(z)$ and $G_{0,2}(z)$. 
Note that, as~\eqref{eq:GGGG01} is
 not defined on the negative real axis, we apply analytic continuation in order to derive an expression 
which is defined for every $|z|<\frac{1}{4}$, which is the radius of convergence of $G_{0,1}(z)$. 
The function $G_{0,1}(z)$ is an algebraic function since it is the sum of two algebraic functions (namely, $u_1(z)$ and $u_2(z)$).
Using a computer algebra package, it is easy to derive an algebraic equation for $G_{0,1}(z)$. 
For example, the following {\sl Maple} commands (see~\cite{SalvyZimmermann94} for more on these aspects) gives the desired equation:

\medskip
\begin{maplegroup}
\begin{mapleinput}
\mapleinline{active}{1d}{AllRoots:=allvalues(solve(1-z*P(u),u)):}{}\vspace{0.3\baselineskip}
\mapleinline{active}{1d}{u1:=AllRoots[2]: u2:=AllRoots[3]:}{}\vspace{0.3\baselineskip}
\mapleinline{active}{1d}{algeq:=algfuntoalgeq(u1+u2,u(z));}{}\vspace{0.3\baselineskip}
\end{mapleinput}
\end{maplegroup}
\vspace{-1.3\baselineskip}
\begin{align}
	z u^4+2z u^3 +(3z-1)u^2+(2z-1)u+z. \label{eq:G01eqalg}
\end{align}

\noindent
In particular, $G_{0,1}(z)$ is uniquely determined by the previous equation 
and the fact that its expansion at $z=0$ is a power series with non-negative coefficients. 
Solving this equation, we arrive at an analytic expression for $G_{0,1}(z)$ for $|z|<1/4$:
\begin{align} 
\label{eq:G01analytic}
	G_{0,1}(z)&= -\frac{1}{2} + \frac{1}{2} \sqrt{\frac{2 - 3z - 2 \sqrt{1 - 4z}}{z}}\\
&=z + z^2 + 3 z^3 + 7 z^4 + 22 z^5 + 65 z^6 + 213 z^7 + \cdots. 
\end{align}

%

\medskip
Using a computer algebra package again, we find that $G_{0,2}(z)$ satisfies
\begin{align} \label{eq:G02eq}
	z^3 u^4 -3 z^2 u^3 - (z^2-3z)u^2+(z-1)u+z=0.
\end{align}
Among its four branches, only one is a power series at $z=0$ with
non-negative coefficients, namely
\begin{align}
\label{eq:G02analytic}
	G_{0,2}(z)&=\frac{3-\sqrt{1-4 z}-\sqrt{2 + 12 z + 2\sqrt{1-4z}}}{4 z}\\
&= z+z^2+4z^3+9z^4+31z^5+ 91z^6+ 309z^7+\cdots.
\end{align} 

\smallskip
In order to undertake a small digression on complexity of computation:
these explicit forms are not the fastest way to access the coefficients.  
A better way is to take advantage of the theory of holonomic functions 
(as, e.g., implemented in the {\sl gfun} {\sl Maple} package,
see~\cite{SalvyZimmermann94}).  
To begin with, the kernel method gave us an algebraic equation. 
Applying the
derivative to both sides of this equation
and using the obtained new relations, 
we are led  
to a linear differential equation satisfied by the function $G(z)$
(where we write $G(z)$ instead of $G_{0,1}(z)$ for short): 
\begin{mapleinput}\mapleinline{active}{1d}{diffeq:=algeqtodiffeq(subs(u=G,algeq),G(z),{G(0)=0}):}{}\end{mapleinput}
$$\begin{cases}
 G(0)=0,\\ \vspace{0.9\baselineskip}
 6\,z+6+12\, \left( z+1 \right) G \left( z \right) +2\, \left( 162\,{z}^{3}+66\,{z}^{2}+z-3 \right) {\frac {d}{dz}}G \left( z \right) \\ \vspace{0.6\baselineskip}
+z \left( 9\,z+4 \right)  \left( 4\,z-1 \right)  \left( 6\,z+1 \right) { \frac {d^{2}}{d{z}^{2}}}G \left( z \right) =0 \label{eq:G01diffeq}
\end{cases}$$

\noindent
Then, extraction of $[z^n]$ on both sides of the differential equation 
yields a linear recurrence satisfied by the coefficients $g(n)$ of
$G$, namely
\begin{mapleinput}\mapleinline{active}{1d}{rec:=diffeqtorec(diffeq,G(z),g(n)):}{}\end{mapleinput}
$$\begin{cases}
  g(0) = 0, g(1) = 1, g(2) = 1\,,\\ \vspace{0.6\baselineskip}
0=108\,n \left( 2\,n+1 \right) g \left( n \right) +6\, \left( 13\,{n}^{2}+35\,n+24 \right) g \left( n+1 \right) \\ \vspace{0.3\baselineskip}
  - \left( 17\,{n}^{2}+49\,n+18 \right) g \left( n+2 \right) -2\, \left( 2\,n+7 \right)  \left( n+3 \right) g \left( n+3 \right) \,. \label{eq:G01rec}
\end{cases}$$

\smallskip
\noindent
From this recurrence, 
a binary splitting approach introduced by the Chudnovskys
gives a procedure which surprisingly computes $g(n)$ in only $O(\sqrt{n})$ operations (and $O(n \ln n \ln(\ln n ))$ bit complexity):
\begin{mapleinput}
\mapleinline{active}{1d}{g:=rectoproc(rec,g(n)):}{}\vspace{0.3\baselineskip}
\mapleinline{active}{1d}{g(10^5):  #a 6014-digits number computed in only 2 seconds!}{}
\end{mapleinput}

\medskip
The same approach applies to all our directed lattice path
models. This approach is much faster than the naive approach by means of dynamic
programming 
(which would compute the bivariate generating function, and would then
extract the desired $G(z)$ from it: this would cost $O(n^2)$ in time
and $O(n^3)$ in memory). 

\medskip
We just saw how to efficiently compute $g(n)$, for any given value of
$n$, but is there a closed-form formula holding for all $n$ at once? 
We now further investigate this question. 

\subsection{How to get a closed form for coefficients: Lagrange--B\"urmann inversion}
In Section~\ref{Section4}, we present a closed form for the numbers of
lattice walks with step polynomial
$P(u)=u^{-h}+u^{-h+1}+\dots+u^{h-1}+u^h$, for any $h$. 
In the case $h=2$ that we are dealing with in the current section, 
a nice miracle occurs: a more ad hoc approach
allows one to derive simpler expressions.

\subsubsection{Closed form for coefficients of $G_{0,1}(z)$}

The generating function $G_{0,1}(z)$ of walks starting at  the origin, ending at altitude~$1$, and never touching the $x$-axis,
satisfies the algebraic equation \eqref{eq:G01eqalg}. 
We rewrite it in the form
\begin{align*}
	G_{0,1}(z) + G_{0,1}^2(z) = z(1 + G_{0,1}(z) + G_{0,1}^2(z))^2.
\end{align*}
Here, substitution of $G_{0,1}(z) + G_{0,1}^2(z)$ by $C(z)-1$ gives the 
striking equation
\begin{align}
	\label{eq:LagrangeG01C}
	1+ G_{0,1}(z)+G_{0,1}^2(z)  &= C(z) ,
\end{align}
where $C(z)=1+zC(z)^2$ is the generating function for Catalan numbers. 
A recursive bijection for this identity was found by Axel Bacher and
(independently) by J\'er\'emie Bettinelli and \'Eric Fusy (personal
communication, see also~\cite{BettinelliFusy16}). 
It remains a challenge to find a more direct simple bijection.
This identity is the key to get nice closed-form expressions for the coefficients, via the following variant of Lagrange inversion. 

\begin{lemma}[\sc Lagrange--B\"urmann inversion variant] 
\label{lemma:LagrangeComp}
	Let $F(z)$ and $H(z)$ be two formal power series
satisfying the equations
	\begin{align}
		F(z) &= z \phi(F(z)), &
		H(z) &= z \psi(H(z)),
	\end{align}
	where $\phi(z)$ and $\psi(z)$ are formal power series such that $\phi(0) \neq 0$ and $\psi(0) \neq 0$.	Then,
	\begin{align}\label{LagrangeComp}
		[z^n]H(F(z)) &= \frac{1}{n} \sum_{k=1}^{n} 
\left([z^{k-1}] \psi^{k}(z)\right)\left( [z^{n-k}] \phi^n(z)\right).
	\end{align}
\end{lemma}
\begin{proof}
By the Lagrange--B\"urmann inversion (Theorem~\ref{thm:LagrangeBurmann}),
we have
$$
	[z^n]H(F(z)) = \frac{1}{n} [z^{n-1}] H'(z) \phi^n(z).
$$
Now we apply the Cauchy product formula 
$[z^m] A(z) B(z) = \sum_{k=0}^m  a_k b_{m-k}$
with $m=n-1$, $A(z)=H'(z)$, and $B(z)= \phi^n(z)$. This leads to
\begin{align*}
[z^n]H(F(z))
	&= \frac{1}{n} \sum_{k=0}^{n-1} 
\left([z^k] H'(z)\right)\left(  [z^{n-1-k}] \phi^n(z)\right)   \\
&= \frac{1}{n} \sum_{k=1}^{n} 
\left([z^{k-1}] H'(z)\right)\left(  [z^{n-k}] \phi^n(z) \right).
\end{align*}
This gives Formula~\eqref{LagrangeComp}, after observing
   $
		[z^{k-1}]H'(z) = k [z^{k}] H(z) = [z^{k-1}] \psi^{k}(z)\,,
	$	
	where we used Lagrange--B\"urmann  inversion again. 
	\end{proof}

\begin{proposition}\label{prop:G011}
The number of basketball walks of length $n$ from the origin to
altitude~$1$ with steps in $\S =\{-2,-1,+1,+2\}$ 
and never returning to the $x$-axis equals 
	\begin{equation}\label{eq:G011}
 \frac{1}{n} \sum_{k=1}^{n} (-1)^{k+1} \dbinom{2k-2}{k-1} \dbinom{2n}{n-k}\\
		= \frac{1}{n}\sum_{i=0}^{n}\binom{n}{i}\binom{n}{2n+1-3i}\,. 
	\end{equation}
\end{proposition}
\begin{proof}
Equation~\eqref{eq:LagrangeG01C} implies that
	$
	G_{0,1}(z) = H(C(z)-1),
	$
where $H(z)$ is the functional inverse of the polynomial $x^2+x$.
Thus $H(z) = z \psi(H(z))$, with $\psi(z) = \frac{1}{1+z}$. 
Furthermore, it is well-known that $C_0(z):=C(z)-1$ satisfies $C_0(z) = z \phi(C_0(z))$ with $\phi(z) = (1+z)^2$. 
Hence, Equation~\eqref{LagrangeComp} yields
	\begin{align}
	[z^n] G_{0,1}(z) &= \frac{1}{n} \sum_{k=1}^{n} 
\left([z^{k-1}] \frac{1}{(1+z)^k}\right)\left( [z^{n-k}] (1+z)^{2n}\right)\\
	&= \frac{1}{n} \sum_{k=1}^{n} (-1)^{k+1} \dbinom{2k-2}{k-1} \dbinom{2n}{n-k}\,.
	\end{align}
	
	The alternative expression without the $(-1)^{k+1}$ factors comes from 
	Formula~\eqref{eq:GGGG01}, 
	to which we apply the Lagrange--B\"urmann inversion formula for $u_1$, remembering that $u_1$ satisfies $u^2 = zu^2 P(u)$, and that the conjugation property of the small roots from Proposition~\ref{prop:conjugation} holds:
		\begin{align}
		[z^n] G_{0,1}(z)&=[z^n] (u_1(z)+u_2(z))
		=2[z^{n}] u_1(z)
		=\frac{1}{n}\sum_{k=0}^{n}\binom{n}{k}\binom{n}{2n+1-3k}\,. \qedhere
	\end{align}
\end{proof}


The last closed-form expression can also be explained 
via the so-called cycle lemma (cf.\ \cite[Ex.~5.3.8]{stan99}). Namely,
by~\eqref{eq:unconstrainedpaths} combined with the factorization
$u^{-2}+u^{-1}+u+u^2=u^{-2}(1+u^3)(1+u)$,
the number of unrestricted walks from $0$ to $1$ in $n$ steps is given
by 
\begin{align}
	[u^1z^n] W(z,u) = [u^1] P(u)^n = [u^1] \left(\frac{(1 + u^3) (1 + u) }{u^2} \right)^n
	=\sum_{i=0}^n\binom{n}{i}\binom{n}{2n+1-3i}.
\end{align}

\begin{figure}[!hb]
	\centering
	\includegraphics[scale=1.05]{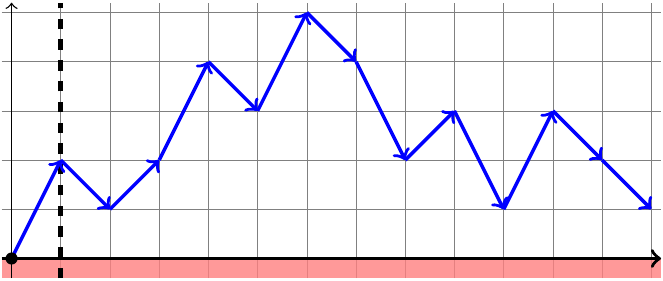} \qquad
	\includegraphics[scale=1.05]{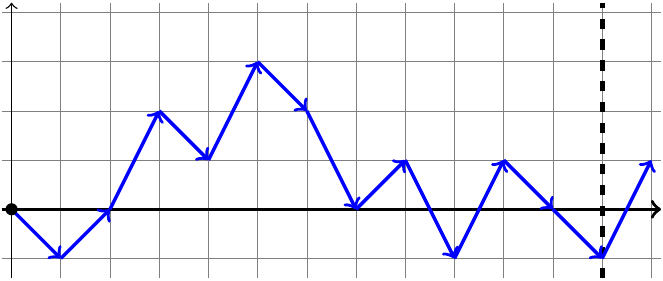}
	\caption{Transforming a walk counted by $G_{0,1}(z)$ 
into a walk counted by $W_{0,1}(z)$. }
	\label{fig:G01cycle}
\end{figure} 

From the formulas, we see that
$[z^n]G_{0,1}(z)=\frac{1}{n}[z^n]W_{0,1}(z)$. There exists indeed a
$1$-to-$n$ correspondence between walks counted by $G_{0,1}(z)$ and
those counted by $W_{0,1}(z)$.
	For each walk $\omega$ counted by $G_{0,1}(z)$, decompose $\omega$ into 
	$\omega=\omega_{\ell}B\omega_{r}$ 
	where $B$ is any point in the walk.
A new walk $\omega'$ counted by $W_{0,1}(z)$ is constructed by putting $B$ at the origin and adjoining $\omega_\ell$ at the end of $\omega_r$, i.e.,
	 $\omega'=B\omega_r\omega_\ell$, see Figure~\ref{fig:G01cycle}. 
	If $\omega$ is of length $n$, then there are $n$ choices for
        $B$. All these walks are different because there are no
        walks from altitude $0$ to altitude $1$ which are the concatenation
        of several copies of one and the same walk. 
(This is not true for walks from altitude~$0$ to 
altitude~$2$. For example, the walk $(0,2,1,3,2)$ is the 
concatenation of two copies of the walk $(0,2,1)$.)  
	
	Conversely, given a walk $\tau$ of length $n$ counted by
        $W_{0,1}(z)$, we decompose $\tau$ into 
	$\tau=\tau_{\ell}B\tau_{r},$
	where $B$ is the right-most minimum of $\tau$. Then,
	$\tau'=B\tau_{r}\tau_{\ell}$
	is a walk of length $n$ counted by $G_{0,1}(z)$. 
	

\subsubsection{Closed form for the coefficients of $G_{0,2}(z)$}

Recall that, by means of the kernel method, we derived a closed form
expression for the generating function $G_{0,2}(z)$ in 
\eqref{eq:G02analytic}.

\begin{proposition}\label{prop:G012}
The number of basketball walks of length $n$ from the origin to altitude~$2$ with steps in $\S =\{-2,-1,+1,+2\}$ and never returning to the $x$-axis equals 
	\begin{equation}
 \frac{1}{2n+1} \sum_{k=0}^{n+1} (-1)^{n+k+1} \dbinom{2n+1}{n+k} \dbinom{n+2k-1}{k}\,. \label{eq:G0222} 
	\end{equation}
\end{proposition}

\begin{proof}
We define the series $F(z)$ by
\begin{equation} \label{eq:G02F} 
-\frac {1} {F(z)}=G_{0,2}(z)-\frac {1} {z}.
\end{equation}
It is straightforward to see from this equation that
$F(z)=z+z^3+\cdots$. The equation~\eqref{eq:G02eq} translates into
the equation
\begin{equation} \label{eq:Feq} 
(F^3(z)-zF(z))(1+F(z))+z^2=0
\end{equation}
for $F(z)$. We may rewrite this equation in the form
$$
\left(F^2-\frac {z} {2}\right)^2
=
\frac {z^2} {4}\cdot\frac {1-3F(z)} {1+F(z)}.
$$
Next we take the square root on both sides. In order to decide the
sign, we have to observe that $F^2(z)=z^2+\cdots$, hence
$$
F^2(z)-\frac {z} {2}
=
-\frac {z} {2}\sqrt{\frac {1-3F(z)} {1+F(z)}},
$$
or, equivalently, $F(z)$ satisfies $F^2(z)=zB(F(z))$, where
$$
B(z)=\frac {1} {2}\left(1
-\sqrt{\frac {1-3z} {1+z}}\right).
$$
It is straightforward to verify that $B(z)$ satisfies the equation
$B(z) = z A(B(z))$ with $A(z) = \frac{1}{1-z} - z$, and it is the
only power series solution of this equation. Hence,
for $n\ge1$, by \eqref{eq:G02F}, Lagrange--B\"urmann inversion 
(Theorem~\ref{thm:LagrangeBurmann}) with $H(z)=z^{-1}$, we have
\begin{equation*}
[z^n] G_{0,2}(z) = -[z^n]\frac {1} {F(z)}
=\frac {1} {n}[z^{n-1}] z^{-2}\left(\frac {B(z)} {z}\right)^n
=\frac {1} {n}[z^{2n+1}] B^n(z).
\end{equation*}
Now we apply Lagrange--B\"urmann inversion again, this time with
$F(z)$ replaced by $B(z)$, $n$ replaced by $2n+1$, 
and $H(z)=z^n$. This yields
\begin{align}
[z^n] G_{0,2}(z) &=
 \frac{1}{n(2n+1)} [z^{2n}] n z^{n-1} A^{2n+1}(z) \\
	        &= \frac{1}{2n+1} [z^{n+1}] \left(\frac{1}{1-z} - z\right)^{2n+1}\,.
\end{align}
By applying the binomial theorem, we then obtain
\begin{align}
	[z^n] G_{0,2}(z) &= \frac{1}{2n+1} [z^{n+1}] \sum_{k=0}^{2n+1} (-1)^{k+1} \dbinom{2n+1}{k} z^{2n+1-k} \left(\frac{1}{1-z}\right)^{k}.
\end{align} 
Since
\begin{align}
	\left(\frac{1}{1-z}\right)^{k} &= \sum_{\ell \geq 0} \dbinom{k+\ell-1}{\ell}z^{\ell}\,,
\end{align}
we get
\begin{align}
	[z^n] G_{0,2}(z) &= \frac{1}{2n+1} [z^{n+1}] \sum_{\ell \geq 0} \sum_{k=0}^{2n+1} (-1)^{k+1} \dbinom{2n+1}{k} \dbinom{k+\ell-1}{\ell} z^{2n+1-k+\ell}\\
	     &= \frac{1}{2n+1} \sum_{k=n}^{2n+1} (-1)^{k+1} \dbinom{2n+1}{k} \dbinom{2k-n-1}{k-n}\\
	     &= \frac{1}{2n+1} \sum_{k=0}^{n+1} (-1)^{n+k+1} \dbinom{2n+1}{n+k} \dbinom{n+2k-1}{k}\,,\label{eq:G02}
\end{align}
as desired.
\end{proof}

The idea of the above proof was to ``build up'' a chain of dependencies
between the actual series of interest, $G_{0,2}(z)$, and several
auxiliary series, namely the series $F(z)$, $B(z)$, and $A(z)$, so
that repeated application of Lagrange--B\"urmann inversion could be
applied to provide an explicit expression for the coefficients of the
series of interest.
This raises the question whether this example is just a coincidence,
or whether there exists a general method to transform a power series
into a Laurent series with the same positive part, and a ``nice''
algebraic expression, allowing multiple Lagrange--B\"urmann inversions to get
``nice'' closed forms for the coefficients. We have no answer to this
question and therefore leave this to future research.

\subsubsection{Closed form for the coefficients of basketball excursions}

Here, we enumerate basketball {\it excursions}, that is, basketball
walks which start at the origin, return to altitude~$0$, and in
between do not pass below the $x$-axis. A main difference to the
previously considered {\it positive} basketball walks is that
the excursions are allowed to touch the $x$-axis anywhere.

\begin{proposition}[\sc Enumeration of basketball excursions]
\label{prop:baskexc}
The number of basketball walks with steps in $\S =\{-2,-1,+1,+2\}$
of length $n$ from the origin to altitude $0$ never passing below
the $x$-axis is
	\begin{equation}
e_n := \frac{1}{n+1}\sum_{k=0}^{n}(-1)^{n+k}\binom{2n+2}{n-k}\binom{n+2k+1}{k}
		=\frac{1}{n+1}\sum_{i=0}^{\lfloor n/2 \rfloor}\binom{2n+2}{i}\binom{n-i-1}{n-2i}. \label{eq2}
	\end{equation}
\end{proposition}
\begin{remark}
The first few values of the sequence defined by \eqref{eq2} are
$$	1, 0, 2, 2, 11, 24, 93, 272, 971, 3194, 11293, 39148, 139687, 497756,\dots$$
\end{remark}

\begin{proof}[Proof of Proposition \ref{prop:baskexc}]
	By the kernel method, we know that the generating function
        for excursions, $E(z)$ say, is given by
        $E(z)=-\frac{u_1u_2}{z}$, and that it satisfies the algebraic equation 
	\begin{align} 
			&z^4 E^4-(2z^3+z^2)E^3+(3z^2+2z)E^2-(2z+1)E+1=0\, \label{eq:poly.for.excursion}.
	\end{align}
Among the branches of this algebraic equation, only one has a power
series expansion. The equation may be rewritten in the form
$$
zE(z)= z\left(\frac 1{(1-zE(z))^2} -\frac{ 2zE(z)}{1-zE(z)} +
z^2E^2(z)\right)
=
 z\left(\frac 1{1-zE(z)} - zE(z)\right)^2.
$$
This shows that we may apply Lagrange--B\"urmann inversion 
(Theorem~\ref{thm:LagrangeBurmann}) with $\phi(z)= (\frac{1}{1-z}-z)^2$.
So we have
	\begin{align}
	[z^n]E(z)&
	=\frac{1}{n+1}[z^{n}]\phi^{n+1}(z)
	=\frac{1}{n+1}[z^{n}]\left(\frac{1}{1-z}-z\right)^{2n+2}\\
	&=\frac{1}{n+1}\sum_{k=0}^{n}(-1)^{n+k}\binom{2n+2}{n-k}\binom{n+2k+1}{k}.
	\end{align}
	It is possible to get an expression involving only positive
        summands by making use of the rewriting 
$\phi(z)=(1+\frac{z^2}{1-z})^2$. This leads to~\eqref{eq2}.
\end{proof}

The trick used in this proof can in fact 
be translated into an algorithm of wider use:

\leavevmode
\begin{center}
\fbox{\begin{minipage}{0.9\textwidth}
{\tt The ``Lagrangean scheme'' algorithm }\medskip

{\tt input:} an algebraic power series 
(given in terms of its algebraic equation $P(z,F)=0$, plus the first terms of the expansion of $F$, so that we can uniquely identify the correct branch of the equation) \medskip


{\tt output:} a ``Lagrangean equation'' satisfied by $F$\newline
(i.e., $H(z^aF) = z \phi(z^aF)$, where $z^a F$ has valuation\footnote{The valuation of a power series $\sum_{n\ge0}f_nz^n$ is the least~$n$ such that  $f_n\ne0$.}~1.)\medskip

{\tt way to process:} if we assume that $H=H_1/H_2$ and $\phi=\phi_1/\phi_2$ are rational functions,
then we identify them via an indeterminate coefficient approach, by 
substituting the {\em{polynomials}} $H_1, H_2, \phi_1, \phi_2$ in the equation $P(z,F)=0$.
\end{minipage}}
\end{center}

\medskip

This algorithm therefore provides a way to get multiple-binomial-sum representations.
See~\cite{Egorychev84,Xin04,BostanLairezSalvy15} for other approaches not relying on the algebraic nature of $F$, 
but designed for the class of functions which can be written as diagonals of rational functions
(these two classes coincide in the bivariate case).
For example, Formula~\eqref{eq2} for $e_n$ has the following alternative representation:
\begin{equation} (n+1) e_n=[t^n] \operatorname{diagonal} \left(
\frac{(1+u)^6 u t^2}{1- ( u (u+1)^2 t+u(1+u)^4  t^2)}+(u+1)^2
\right)\,. \end{equation} 
The rational function on the right-hand side has the striking feature that its bivariate
series expansion has only non-negative coefficients.
In fact, it is even a bivariate $\N$-rational function (i.e., a
function obtained as iteration of addition, multiplication, and
quasi-inverse,\footnote{The {\it quasi-inverse} of a power series
$f(z)$ of positive valuation is $1/(1-f(z))$.} 
starting from polynomials in $u$ and $t$ with positive integer coefficients). 
Given a multivariate rational function, it is a hard task to write it
as an $\N$-rational expression (an algorithm is known in the
univariate case), 
so some human computations were needed here to get the above expression. 
\smallskip

In fact (and we believe that it was not observed before), these
multivariate rational functions appearing in the computation of
diagonals related to nested sums of binomials 
are always $\N$-rational: this follows from the closure properties of
$\N$-rational functions. It is an open question to give a
combinatorial interpretation  
(in terms of the initial structure counted by the diagonal) of the
other diagonals of this rational function. 
It is also not easy to extrapolate from this rational function a
general pattern which could appear for more general sets of steps: 
we shall see in Section~\ref{Section4} which type of formulas generalize the
rich combinatorics that we had for $P(u)=u^{-2}+u^{-1}+u+u^2$.

\subsection{How to derive the corresponding asymptotics: singularity analysis}

We close this section by briefly addressing how to find the
asymptotics of numbers of basketball walks. Indeed,
standard techniques from singularity analysis suffice to get the
asymptotic growth of the coefficients of $z^n$ in the generating
functions that we consider here for $n \to \infty$. 
The interested reader is referred to~\cite{FlSe09} for more details on this subject (see Figure VI.7 therein for an illustration of singularity analysis).

\begin{theorem}
	Let $G_{0,1}(z)$ and $G_{0,2}(z)$ be the generating functions for positive basketball walks with steps $-2,-1,+1,+2$ starting at the origin and ending at altitude $1$, respectively at $2$. Then, as $n \to \infty$, the coefficients are asymptotically equal to
	\begin{align}
	[z^n]G_{0,1}(z) &= \frac{1}{\sqrt{5 \pi}} \frac{4^n}{n^{3/2}} \left( 1 - \frac{81}{200} \frac{1}{n} + O \left(\frac{1}{n^2} \right) \right)\,, \\
	[z^n] G_{0,2}(z) &= \frac{5+\sqrt{5}}{10 \sqrt{\pi}} \frac{4^n}{n^{3/2}} \left( 1 - \frac{201+24\sqrt{5}}{200} \frac{1}{n} + O \left(\frac{1}{n^2} \right) \right)\,.
	\end{align}
\end{theorem}

\begin{proof}
	The asymptotic growth of the coefficients is governed by the
        location of the dominant singularity (the singularity closest
        to the origin). The dominant singularity 
of~\eqref{eq:G01analytic} and~\eqref{eq:G02analytic} is
        given by $1/4$, since the square root becomes singular at this point.
	
	Next, we compute the singular expansion for $z\to 1/4$, which is
a Puiseux series: 
	\begin{align}
	G_{0,1}(z) &= -\frac{1-\sqrt{5}}{2} - \frac{2}{\sqrt{5}} \sqrt{1-4z} + O\left(1-4z\right),\\
	G_{0,2}(z) &= \left(3-\sqrt{5}\right) - \frac{5+\sqrt{5}}{5} \sqrt{1-4z} + O\left(1-4z\right).
	\end{align}
	Finally, we apply the standard function scale from~\cite[Theorem~VI.1]{FlSe09} and the transfer for the error term~\cite[Theorem~VI.3]{FlSe09} to get the asymptotics. 
\end{proof}

More generally, asymptotics for the number of walks from altitude $i$ to altitude $j$ in $n$ steps can be obtained via singularity analysis of the small roots,
similarly to what was done in~\cite{BaFl02}.
Note that it is easy derive as many terms as needed in the asymptotic expansion of the coefficients
by including more terms in the Puiseux expansion.
We also want to point out that this process was implemented 
in {\sl SageMath} (see~\cite{Kauers15}) or in {\sl
  Maple} by Bruno Salvy (as a part of the {\tt algolib} package). There, the {\tt
  equivalent} command directly gives the above result:
\begin{maplegroup}
\begin{mapleinput}
\mapleinline{active}{1d}{equivalent(G01,z,n,3);
}{}
\end{mapleinput}
\mapleresult
\begin{maplelatex}
\mapleinline{inert}{2d}{}{\[\displaystyle \frac{1}{5}\,{\frac { \sqrt{5}\, \, {4}^{n}}{ \sqrt{\pi }\, {n}^{3/2}}}-{\frac {81}{1000}}\,{\frac { \sqrt{5}\, \, {4}^{n}}{ \sqrt{\pi }\, {n}^{5/2}}}\\
\mbox{}+O \left( {\frac {{4}^{n}}{{n}^{7/2}}} \right) \,.\]}
\end{maplelatex}
\end{maplegroup}

\section{General case: Lattice walks with arbitrary steps}\label{Section4}

We first prove a theorem which holds for any symmetric set of steps, i.e., when the step polynomial satisfies $P(u) = P(1/u)$.

\begin{theorem}[\sc Positive walk enumeration]
	\label{prop:G0kgeneral} 
Consider walks with a symmetric step polynomial $P(u)$.
	Let $G_{0,k}(z)$ be the generating function for positive
        walks,
i.e., walks starting at the origin, ending at altitude~$k$, and always
staying strictly above the $x$-axis in-between,
	and let $M_{> 0}(z)$ be the generating function of positive 
meanders, i.e., positive walks ending at any altitude $>0$. 	
Then
	\begin{align} 
			M_{> 0}(z)  &= \sum_{k>0} G_{0,k}(z)=\prod_{i=1}^h \frac{1}{1-u_i(z)}, \label{eq:posmeanders}\\
		G_{0,k}(z) &= h_k\left(u_1(z),u_2(z), \ldots, u_h(z) \right), \label{eq:G0kgeneral} 
	\end{align}		
	where $u_1(z),u_2(z),\dots,u_h(z)$ are the small roots of
        the kernel equation $1-zP(u)=0$, and
	\begin{align*} 
		h_k(x_1,x_2,\ldots,x_h) &= 
\underset{i_1 +\dots+ i_h =k}{\sum_{i_1,\dots,i_h\ge0}}
 x_1^{i_1} x_2^{i_2} \cdots x_h^{i_h}
	\end{align*}
	is the complete homogeneous symmetric polynomial of degree $k$ in the variables $x_1,x_2,\break\ldots,x_h$. 
\end{theorem}

\begin{proof}
The formula for positive meanders follows from the expression for meanders (which are allowed to touch the $x$-axis!)	in~\cite[Corollary~1]{BaFl02},
	\begin{align*}
M_{\geq 0 }(z)=		-\frac{1}{z}\prod_{i=1}^h \frac{1}{1-v_i(z)}\,,
	\end{align*}
where $v_1(z),v_2(z),\dots,v_h(z)$ are the large roots of $1-zP(u)=0$,
i.e., those roots $v(z)$ for which $\lim_{z\to0}\vert v(z)\vert=\infty$.
Every meander starts with an initial excursion, and later never
returns to the $x$-axis any more. This simple fact implies the generating
function equation $M_{\geq 0}(z)=E(z) M_{>0}(z)$. Hence, we need
to divide the above expression for $M_{\ge0}(z)$ by the generating function for
excursions --- which, by~\cite[Theorem~2]{BaFl02}, is given by
$$E(z) = \frac{(-1)^{h-1}}{z} \prod_{i=1}^h u_i(z).$$ 
Finally, due to $P(u)=P(u^{-1})$, we have $u_i(z) = 1/v_i(z)$, which gives the final expression for $M_{>0}$, while the formula for $G_{0,k}(z)$ is proven in~\cite{BaWa16b}. 
\end{proof}

This proof shows, in particular, that generating functions for strictly
positive walks, respectively for weakly positive walks, are intimately
related, and are therefore given by similar expressions. 
(The price of positivity is a division by $E(z)$, which encodes the
excursion prefactor.) 
The proof also extends to non-symmetric steps, but then the formulas involve one more factor. It is possible to deal with them exactly in the way we proceed for symmetric steps, but this leads to slightly less nice formulas.

\medskip
In the sequel, we focus on positive walks with symmetric steps.
We show in which way we can use the obtained expressions for the generating functions in order 
to get nice closed-form expressions for their coefficients.

\subsection{Counting walks with steps in  \texorpdfstring{$\S=\{0,\pm1,\dots,\pm h\}$}{S=\{-h,-h+1,\dots,h-1,h\}}}

In Section~\ref{Section3} on basketball walks, we had a taste of what the kernel method could do for us when combined with Lagrange--B\"urmann inversion. This was, however, only for the case $\mathcal{S}=\{\pm1,\pm2\}$. In this section, we illustrate again the power of the kernel method, when applied to more general step sets~$\mathcal{S}$. We first start with a generalization of Section~\ref{Section2} to $\mathcal{S}=\{0,\pm1,\dots,\pm h\}$.
In order to have a convenient notation, we introduce
{\em $m$-nomial coefficients} by defining
\begin{equation}
 \binom{n}{k}_m := [ u^k ]	(1+u+\cdots+u^{m-1})^n\,,
\end{equation}
where $k$ is between $0$ and $(m-1) n$.

%

\begin{proposition} 
The $m$-nomial coefficient equals
\begin{equation}\label{mnomial}
	\binom{n}{k}_m=\sum_{i=1}^n(-1)^i\binom{n}{i}\binom{n+k-mi-1}{n-1}.
\end{equation}
\end{proposition}
\begin{proof}
Coefficient extraction in the defining expression for $\binom nk_m$ yields
\begin{align}
	\binom{n}{k}_m
	&=[u^k](1+u+\cdots+u^{m-1})^n
	=[u^k](1-u^{m})^n\dfrac{1}{(1-u)^n}\\
	&=[u^k]\left( \sum_{i=0}^n\binom{n}{i}(-1)^iu^{mi} \right) 
	  \left( \sum_{j\ge0}^{}\binom{n+j-1}{n-1}u^j\right) \\
	&=\sum_{i=0}^{\left\lfloor (n+k-1)/m \right\rfloor} (-1)^i\binom{n}{i}\binom{n+k-mi-1}{n-1}. 
	\end{align}
	The upper bound in the sum can be taken more naturally to be $i=n$, using the convention that binomials $\binom{n}{k}$
	are $0$ for $n<0$ or $k>n$ (the reader should be warned that this is not the convention of {\sl  Maple} or {\sl Mathematica}).  
	This gives Formula~\eqref{mnomial}.
\end{proof}

{\sl Historical remark.} These $m$-nomial coefficients appear in more than fifty articles (many of them focusing on trinomial coefficients) dealing with their rich combinatorial aspects 
(see e.g.~\cite{AndrewsBaxter87,Renzi92,Andrews07,Blaziak08}). We use
the notation $\binom{n}{k}_m$ promoted by George
Andrews~\cite{Andrews90}. It should be noted that
they were previously called {\it polynomial coefficients} 
by Louis Comtet~\cite[p.~78] {Comtet74}, who is mentioning early work
of D\'esir\'e Andr\'e (with a typo in the date)  
and Paul Montel~\cite{Andre1876,Montel1942},
and who was himself using another notation for these numbers, namely $\binom{n,m}{k}$. 

\smallskip These coefficients have a direct combinatorial interpretation in terms of lattice walk enumeration.

\begin{theorem}[\sc Unconstrained walk enumeration]
The number of unconstrained\/\footnote{Unconstrained means that
  the walks are allowed to have both positive and negative altitudes.}
walks running from the origin to altitude $k$ in $n$ steps taken from 
$\{0,\pm1,\pm2,\break\dots,\pm h\}$ equals $\binom{n}{k+hn}_{2h+1}$.
\end{theorem}
\begin{proof}
By~\eqref{eq:unconstrainedpaths},
the generating function for unconstrained walks is $$W(z,u)=\frac{1}{1-zP(u)} = \sum_{n=0}^\infty P^n(u)z^n\,.$$ 
Then a simple factorization shows that
\begin{align} 
	[u^k]P^n(u)
	&=[u^k]\left(\sum_{i=-h}^{h}u^i\right)^n
	=[u^k]u^{-hn}\left(\sum_{i=0}^{2h}u^i\right)^n=\binom{n}{k+hn}_{2h+1}. \label{eq:steppoly}\qedhere
\end{align}
\end{proof}

Now we will see how to link these coefficients with {\em constrained\/}
lattice walks. To this end, we first state the general version of the
conjugation principle that we encountered in 
Proposition~\ref{prop:conjugation}. 

\begin{proposition}[\sc Conjugation principle for small roots]
\label{prop:conjugationgeneral}
Let 
$$P(u) = \sum_{i=-c}^d p_i u^i$$
be the step polynomial, and let $\omega = e^{ 2 \pi i / c}$ be a
$c$-th root of unity. 
The small roots $u_i(z)$, $i=1,2,\ldots,c$, of\/ $1-zP(u)=0$ 
satisfy
\begin{equation}
u_i(z)=\sum_{n\geq 1} \omega^{n (i-1)} a_n z^{n/c}
\end{equation}
for certain ``universal" coefficients $a_n$, $n=1,2,\dots$.
\end{proposition}
\begin{proof}
The kernel equation yields
$$
u=X \left( p_{-c} + p_{-c+1}u + p_{-c+2} u^2 + \cdots + p_{d-1}u^{c+d-1} + p_{d}u^{c+d} \right)^{1/c},
$$
with $X=\omega^{j} {z}^{1/c}$ for $j=0,1,\ldots,c-1$.  
Since the above equation possesses a unique formal power series solution $u(X)$, the claim follows.
\end{proof}

Next, we apply Lagrange--B\"urmann inversion
to the small roots given by the kernel method, and combine it with the conjugation principle.

\begin{proposition}[\sc Explicit expansion of the roots $u_i$]
\label{prop:coeffuigeneral}
For lattice walks with step polynomial given by 
$P(u)=u^{-h}+u^{-h+1}+\dots+u^{h-1} +u^h$,
let $U(z)$ be the root of\/ $1-z^h P(U) = 0$ whose Taylor
expansion at~$0$ starts $U(z)=z+\cdots$. The series
$U(z)$ is a power series, not a genuine Puiseux series.
Then all small and large roots can be expressed in terms of $U(z)$,
namely we have
\begin{equation}
	u_i(z)=U(\omega^{i-1}z^{1/h}) 
\text{\quad and \quad}	v_i(z)=1/U(\omega^{i-1}z^{1/h}),
\quad i=1,2,\dots,h,
\end{equation}
where $\omega=e^{2\pi i/h}$ is a primitive $h$-th root of unity.
The expansion of a power of the series $U(z)$ is explicitly given by
	\begin{equation}
		U^m(z)=\sum_{n=m}^\infty \frac{m}{n}\binom{n/h}{n-m}_{2h+1}z^n.
	\end{equation}
\end{proposition}

\begin{proof}
We want to solve $1-zP(u)=0$ for $u$. We may rewrite this equation as
$$
z=\frac {u^h} {1+u+\cdots+u^{2h}}.
$$
Taking the $h$-th root, we get
$$
\omega^{i-1}z^{1/h}=\frac {u} {(1+u+\cdots+u^{2h})^{1/h}},
$$
for some $i$ with $1\le i\le h$. 

Since an equation of the form 
$Z=u\phi(u)$, where $\phi(u)$ is a power series in~$u$, 
has a unique power series solution $u(Z)$, the above equation has a unique
solution $u_i(z)$, which turns out to have exactly the form described
in the proposition.
The equation for $v_i$ follows from $u_i=1/v_i$ as we have $P(u)=P(1/u)$.

The equation for $U^m$ comes from Lagrange--B\"urmann inversion:
	\begin{align}
		[z^n]U^m(z)
		&=\frac{1}{n}[z^{-1}](z^m)^\prime P^{n/h}(z)\\
		&=\frac{m}{n}[z^{-m}]\sum_{k}z^k\binom{n/h}{k+n}_{2h+1}\\
		&=\frac{m}{n}\binom{n/h}{n-m}_{2h+1}. \qedhere
	\end{align}
\end{proof}

\begin{theorem}[\sc Closed-form expression for walks with 
$\mathcal{S}=\{0,\pm1,\dots,\pm h\}$]
The numbers of positive walks and meanders from the origin to
altitude~$k$ in $n$~steps from $\mathcal{S}\!=\!\{0,\pm1,\dots,\pm h\}$
admit the closed-form expressions 
	\begin{align*}
		[z^n] G_{0,k}(z) &= 
\sum_{n_1+\dots+n_h=nh}\,\sum_{i_1+\dots+i_h=k}
\frac {i_1} {n_1}{\binom {n_1/h}{n_1-i_1}}_{2h+1}\\
&\kern6cm
\cdots
\frac {i_h} {n_h}{\binom {n_h/h}{n_h-i_h}}_{2h+1}
\omega^{\sum_{j=1}^h(j-1)n_j},
		\\
				[z^n] M_{>0}(z) &= 
\sum_{n_1+\dots+n_h=nh}\,\sum_{i_1,\dots,i_h\ge0}
\frac {i_1} {n_1}{\binom {n_1/h}{n_1-i_1}}_{2h+1}
\cdots
\frac {i_h} {n_h}{\binom {n_h/h}{n_h-i_h}}_{2h+1}
\omega^{\sum_{j=1}^h(j-1)n_j}.
	\end{align*}
		\end{theorem}
	\begin{proof}
We use the expansions from Proposition~\ref{prop:coeffuigeneral} in the generating function formulas from Theorem~\ref{prop:G0kgeneral}.
\end{proof}

%
%

Here are some sequences of numbers of positive walks with steps
$\mathcal{S}=\{0,\pm1,\dots,\pm h\}$, starting at the origin, and
ending at altitude~$1$, for different values of~$h$: 
\begin{align*}
	h &=1 \quad \OEIS{A168049}\footnotemark{}: && 0, 1, 1, 2, 4, 9, 21, 51, 127, 323, 835, \ldots \\	
	h &=2 \quad \OEIS{A104632}: && 0, 1, 2, 6, 20, 73, 281, 1125, 4635, 19525, 83710, \ldots \\	
	h &=3 \quad \OEIS{A276902}: && 0, 1, 3, 12, 56, 284, 1526, 8530, 49106, 289149, 1733347, \ldots \\	
	h &=4 \quad \OEIS{A277920}: && 0, 1, 4, 20, 120, 780, 5382, 38638, 285762, 2162033, 16655167, \ldots 
\end{align*}

\footnotetext{Axxxxxx refers to the corresponding sequence in the
On-Line Encyclopedia of Integer Sequences, available electronically at \url{https://oeis.org}.}
Furthermore,
here are some sequences of numbers of positive walks with steps $\mathcal{S}=\{0,\pm1,\dots,\pm h\}$, starting at the origin, and ending at altitude~$2$, for 
small values of~$h$:
\begin{align*}
	h &=1 \quad \OEIS{A105695}: && 0,0,1,2,5,12,30,76,196,512,1353, \ldots \\	
	h &=2 \quad \OEIS{A276903}: && 0,1,2,7,25,96,382,1567,6575,28096,121847,534953, \ldots \\
	h &=3 \quad \OEIS{A276904}: && 0,1,3,14,68,358,1966,11172,65104,387029,2337919, \ldots \\	
	h &=4 \quad \OEIS{A277921}: && 0,1,4,23,142,950,6662,48420,361378,2753687,21334313,\ldots 
\intertext{\indent Here are the corresponding sequences for positive meanders:}
	h &=1 \quad \OEIS{A005773}: && 1, 1, 2, 5, 13, 35, 96, 267, 750, 2123, 6046, 17303, \ldots \\	  
	h &=2 \quad \OEIS{A278391}: && 1, 2, 7, 29, 126, 565, 2583, 11971, 56038, 264345, \ldots \\
	h &=3 \quad \OEIS{A278392}: && 1,3, 15, 87, 530, 3329, 21316, 138345, 906853,  \ldots \\	
	h &=4 \quad \OEIS{A278393}: && 1,4,26,194,1521,12289,101205,844711,7120398,\ldots 
\intertext{\indent Here are the corresponding sequences for meanders (allowed to touch $0$): }
	h &=1 \quad \OEIS{A005773}: && 1, 2, 5, 13, 35, 96, 267, 750, 2123, 6046, 17303, 49721,  \ldots \\	
	h &=2 \quad \OEIS{A180898}: && 1, 3, 12, 51, 226, 1025, 4724, 22022, 103550, 490191, \ldots \\  
	h &=3 \quad \OEIS{A180899}: && 1, 4, 22, 130, 803, 5085, 32747, 213419, 1403399, \ldots \\	
	h &=4 \quad \OEIS{A180900}: && 1, 5, 35, 265, 2100, 17075, 141246, 1182719, 9994086,\ldots   
\intertext{\indent  Here are the corresponding sequences for excursions:}
	h &=1 \quad \OEIS{A001006}: && 1, 1, 2, 4, 9, 21, 51, 127, 323, 835, 2188, 5798, 15511, 41835, \ldots \\	   
	h &=2 \quad \OEIS{A104184}: && 1, 1, 3, 9, 32, 120, 473, 1925, 8034, 34188, 147787, 647141,    \ldots \\ 
	h &=3 \quad \OEIS{A204208}: && 1, 1, 4, 16, 78, 404, 2208, 12492, 72589, 430569, 2596471,  \ldots \\	
	h &=4 \quad \OEIS{A204209}: && 1, 1, 5, 25, 155, 1025, 7167, 51945, 387000, 2944860,\ldots   
\end{align*}

\begin{remark} Most of the above sequences for $h\geq 3$ were not contained in the On-Line Encyclopedia of Integer Sequences (OEIS) before we added them.
In Section~\ref{Section5}, we discuss the combinatorial structures related to the sequences which were already in the OEIS.
\end{remark}

\subsection{Counting walks with steps in \texorpdfstring{$\mathcal{S}=\{\pm1,\dots,\pm h\}$}{S=\{-h,\dots,-1,1,\dots,h\}}}\label{4.2}

In this Section~\ref{4.2}, we consider the same steps as in the previous one, 
except that we drop the $0$-step.

Certainly, for any type of walks consisting of $k$ steps with $0$-step
included, enumerated by $f_k$ say,
the number of walks of the same type consisting of $n$ steps, all of which
different from the $0$-step, can be obtained by the inclusion-exclusion
principle. The result is $\sum_{k=0}^n(-1)^{n-k}\binom nk f_k$.

Here, our way to derive the corresponding formulas is more ad hoc and relies 
on the shape of the considered steps in~$\S$.
This offers the advantage of leading to positive sum formulas, as opposed to
the alternating sums produced by inclusion-exclusion.
For convenience, we introduce the mock-$m$-nomial coefficients by
	\begin{equation}
 \binom{n}{k}_{2m}^{\hspace{-1mm}*} :=		[u^k] (1+\cdots+u^{m-1}+u^{m+1}+\cdots+u^{2m})^n\,.
	\end{equation}

\begin{proposition}
The mock-$m$-nomial coefficients can be expressed in terms of the
(ordinary) $m$-nomial coefficients in the form\footnote{Here, the
$^{\hspace{0mm}*}$ is a mnemonic to remind us that we do not have the 0 step.}
	\begin{equation}\label{mockmnomial}
	\binom{n}{k}_{2m}^{\hspace{-1mm}*}=\sum_{i=0}^n\binom{n}{i}\binom{n}{k-(m+1)i}_{m}. 
	\end{equation}
\end{proposition}
\begin{proof}
	Factoring the expression and extracting coefficients, we obtain
	\begin{align}
	\binom{n}{k}_{2m}^{\hspace{-1mm}*}&=
	[u^k](1+\cdots+u^{m-1}+u^{m+1}+\cdots+u^{2m})^n\\
	&=[u^k](1+u^{m+1})^n(1+u+\cdots+u^{m-1})^n\\
	&=[u^k]\left(\sum_{i\ge0}\binom{n}{i}u^{(m+1)i} \right)\left(\sum_{j\ge0}\binom{n}{j}_{m}u^j\right)\\
	&=\sum_{i=0}^n\binom{n}{i}\binom{n}{k-(m+1)i}_{m}. \qedhere
	\end{align}
\end{proof}

These mock-$m$-nomial coefficients have also a direct combinatorial 
interpretation in terms of lattice walk enumeration.

\begin{theorem}[\sc Unconstrained walk enumeration] 
The mock-$m$-nomial coefficient
	$\binom{n}{k+hn}_{2h}^{\hspace{-1mm}*}$ is the number of unconstrained walks running from~$0$ to $k$ in $n$ steps taken from $\{\pm1,\pm2,\dots,\pm h\}$.
\end{theorem}

\begin{proof}
We have
	\begin{align}
	[u^k]P^n(u)
	&=[u^k]\left(\sum_{i=-h}^{-1}u^i+\sum_{i=1}^{h}u^i\right)^n \\
	&=[u^k]u^{-hn}\left(\sum_{i=0}^{h-1}u^i+\sum_{i=h+1}^{2h}u^i\right)^n
	=\binom{n}{k+hn}_{2h}^{\hspace{-1mm}*}. \qedhere
	\end{align}
\end{proof}

%


\begin{proposition}[\sc Explicit expansion of the roots $u_i$]	
\label{prop:smallbranches} 
	For lattice walks with step polynomial 
given by $P(u)=u^{-h}+\cdots+u^{-1}+u^1+\cdots+u^{h}$,
	let $U(z)$ be the root of\/ $1-z^h P(U) = 0$ whose Taylor
expansion at~$0$ starts $U(z)=z+\cdots$.
Again, $U(z)$ is a power series, not a genuine Puiseux series.
	Then $U(z)$ satisfies
	\begin{equation}\label{smallroot2}
	U^m(z)=\sum_{n=1}^\infty \frac{m}{n}\binom{n/h}{n-m}^{\hspace{-1mm}*}_{2h}z^n,
	\end{equation}
	and all small and large roots are expressed in terms of $U(z)$ as
	\begin{equation}
	u_i(z)=U(\omega^{i-1}z^{1/h}) 
	\text{\quad and \quad}	v_i(z)=1/U(\omega^{i-1}z^{1/h}),
\quad  \text{for $i=1,2,\dots,h$},
	\end{equation}
	where $\omega=e^{2\pi i/h}$ is a primitive $h$-th root of unity.
\end{proposition}
\begin{proof}
	We apply Lagrange--B\"urmann inversion to get
	\begin{align*}
	[z^n]U^m(z)
	=\frac{1}{n}[z^{-1}](z^m)^\prime P^{n/h}(z)
	=\frac{m}{n}[z^{-m}]\sum_{k}u^k\binom{n/h}{k+n}^{\hspace{-1mm}*}_{2h}
	=\frac{m}{n}\binom{n/h}{n-m}^{\hspace{-1mm}*}_{2h}. \qquad \qedhere
	\end{align*}
\end{proof}

\begin{theorem}[\sc Closed-form expression for walks with 
$\mathcal{S}\!=\!\{\pm1,\dots,\pm h\}$]
The numbers of positive walks and meanders from the origin to
altitude~$k$ in $n$~steps from $\mathcal{S}=\{\pm1,\dots,\pm h\}$
admit the closed-form expressions 
	\begin{align*}
		[z^n] G_{0,k}(z) &= 
\sum_{n_1+\dots+n_h=nh}\,\sum_{i_1+\dots+i_h=k}
\frac {i_1} {n_1}{\binom {n_1/h}{n_1-i_1}}_{2h}^{\hspace{-1mm}*}
\cdots
\frac {i_h} {n_h}{\binom {n_h/h}{n_h-i_h}}_{2h}^{\hspace{-1mm}*}
\omega^{\sum_{j=1}^h(j-1)n_j},\\
				[z^n] M_{>0}(z) &= 
\sum_{n_1+\dots+n_h=nh}\,\sum_{i_1,\dots,i_h\ge0}
\frac {i_1} {n_1}{\binom {n_1/h}{n_1-i_1}}_{2h}^{\hspace{-1mm}*}
\cdots
\frac {i_h} {n_h}{\binom {n_h/h}{n_h-i_h}}_{2h}^{\hspace{-1mm}*}
\omega^{\sum_{j=1}^h(j-1)n_j}.
	\end{align*}
		\end{theorem}
	\begin{proof}
We use the expansions from Proposition~\ref{prop:smallbranches} in the generating function formulas from Theorem~\ref{prop:G0kgeneral}.
\end{proof}

Here are
some sequences of numbers of walks with steps in $\mathcal{S}=\{\pm1,\pm2,\dots,\pm h\}$, starting at the origin, and ending at altitude~$1$, 
for different values of~$h$:
\begin{align*}
	h &=1 \quad \OEIS{A000108}: && 0, 1, 0, 1, 0, 2, 0, 5, 0, 14, 0, \ldots \\	
	h &=2 \quad \OEIS{A166135}: && 0, 1, 1, 3, 7, 22, 65, 213, 693, 2352, 8034, \ldots \\	
	h &=3 \quad \OEIS{A276852}: && 0, 1, 2, 7, 28, 121, 560, 2677, 13230, 66742, 343092, \ldots \\	
	h &=4 \quad \OEIS{A277922}: && 0, 1, 3, 13, 71, 405, 2501, 15923, 104825, 704818, 4827957, \ldots 
\end{align*}

Furthermore, here are 
some sequences of numbers of walks with steps in $\mathcal{S}=\{\pm1,\pm2,\break\dots,\pm h\}$, starting at the origin, and ending at altitude~$2$, 
for different values of~$h$:
\begin{align*}
	h &=1 \quad \OEIS{A000108}: && 0, 0, 1, 0, 2, 0, 5, 0, 14, 0, 42, \ldots \\	
	h &=2 \quad \OEIS{A111160}: && 0,1,1,4,9,31,91,309,1009,3481,11956, \ldots \\	
	h &=3 \quad \OEIS{A276901}: && 0,1,2,9,34,159,730,3579,17762,90538,467796, \ldots \\
	h &=4 \quad \OEIS{A277923}: && 0, 1, 3, 16, 84, 505, 3121, 20180, 133604, 904512, 6224305,\ldots
\end{align*}

Here are the corresponding sequences for positive meanders:
\begin{align*}
	h &=1 \quad \OEIS{A001405}: && 1, 1, 1, 2, 3, 6, 10, 20, 35, 70, 126, 252, 462, 924,  \ldots \\	   
	h &=2 \quad \OEIS{A278394}: && 1, 2, 5, 17, 58, 209, 761, 2823, 10557, 39833, 151147 \ldots \\
	h &=3 \quad \OEIS{A278395}: && 1, 3, 12, 60, 311, 1674, 9173, 51002, 286384, 1620776, \ldots \\	
	h &=4 \quad \OEIS{A278396}: && 1, 4, 22, 146, 1013, 7269, 53156, 394154, 2951950,\ldots 
\end{align*}

Here are the corresponding sequences for meanders (allowed to touch~$0$):
\begin{align*}
	h &=1 \quad \OEIS{A001405}: && 1, 1, 2, 3, 6, 10, 20, 35, 70, 126, 252, 462, 924, 1716, 3432, \ldots \\	   
	h &=2 \quad \OEIS{A047002}: && 1,2,7,23,83,299,1107,4122,15523,58769,223848, \ldots \\
	h &=3 \quad \OEIS{A278398}: && 1,3,15,75,400,2169,11989,66985,377718,2144290, \ldots \\	
	h &=4 \quad \OEIS{A278416}: && 1,4,26,174,1231,8899,65492,487646,3664123,\ldots   
\end{align*}

Here are the corresponding sequences for excursions:
\begin{align*}
	h &=1 \quad \OEIS{A126120}: && 1, 0, 1, 0, 2, 0, 5, 0, 14, 0, 42, 0, 132, 0, 429, 0  \ldots \\	   
	h &=2 \quad \OEIS{A187430}: && 1, 0, 2, 2, 11, 24, 93, 272, 971, 3194, 11293, 39148, 139687 \ldots \\
	h &=3 \quad \OEIS{A205336}: && 1, 0, 3, 6, 35, 138, 689, 3272, 16522, 83792, 434749, \ldots \\	
	h &=4 \quad \OEIS{A205337}: && 1, 0, 4, 12, 82, 454, 2912, 18652, 124299, 841400,\ldots 
\end{align*}

\begin{remark}
The cases with $h=1$ lead to famous sequences, having many links with the
combinatorics of trees, via the {\L}ukasiewicz correspondence (see
Section~\ref{Section2}). 
It is surprising that the cases with $h=2$ also offer many
links with trees, as we show in the next section.  
\end{remark} 

\pagebreak

\section{Some links with other combinatorial problems}\label{Section5}

In this section, we establish some links between our lattice walks and other combinatorial problems.
Thereby we prove several conjectures issued in the On-Line Encyclopedia of Integer Sequences.

\subsection{Trees and basketball walks from 0 to 1}

First, we prove that the sequence \oeis{A166135} from the On-Line Encyclopedia of Integer Sequences,
coming from the enumeration of certain tree structures 
used in financial mathematics,
is in fact related to basketball walks, and corresponds more precisely to the coefficients of $G_{0,1}(z)$.

The $m$-nomial tree is a lattice-based computational model used in financial mathematics to price options.
It was developed by Phelim Boyle~\cite{Boyle86} in 1986.  
For example, for $m=3$, the underlying stock price is modelled as a recombining tree, where, at each node, the price has three possible paths: an up, down, or stable path.
The case $m=2$ has a long history going back to one of the founding problems of financial mathematics and probability theory, the ``ruin problem'', 
analysed in the XVIIIth and XIXth century by de Moivre, Laplace, Huygens, Amp\`ere, Rouch\'e, 
before to be revisited by combinatorialists like Catalan, Whitworth,
Bertrand, Andr\'e, Delannoy (see~\cite{BaSc05} for more on these aspects).
Figure~\ref{4nomialTree} illustrates a 4-nomial tree.

\begin{figure}[!hb]
\includegraphics[width=.6\textwidth]{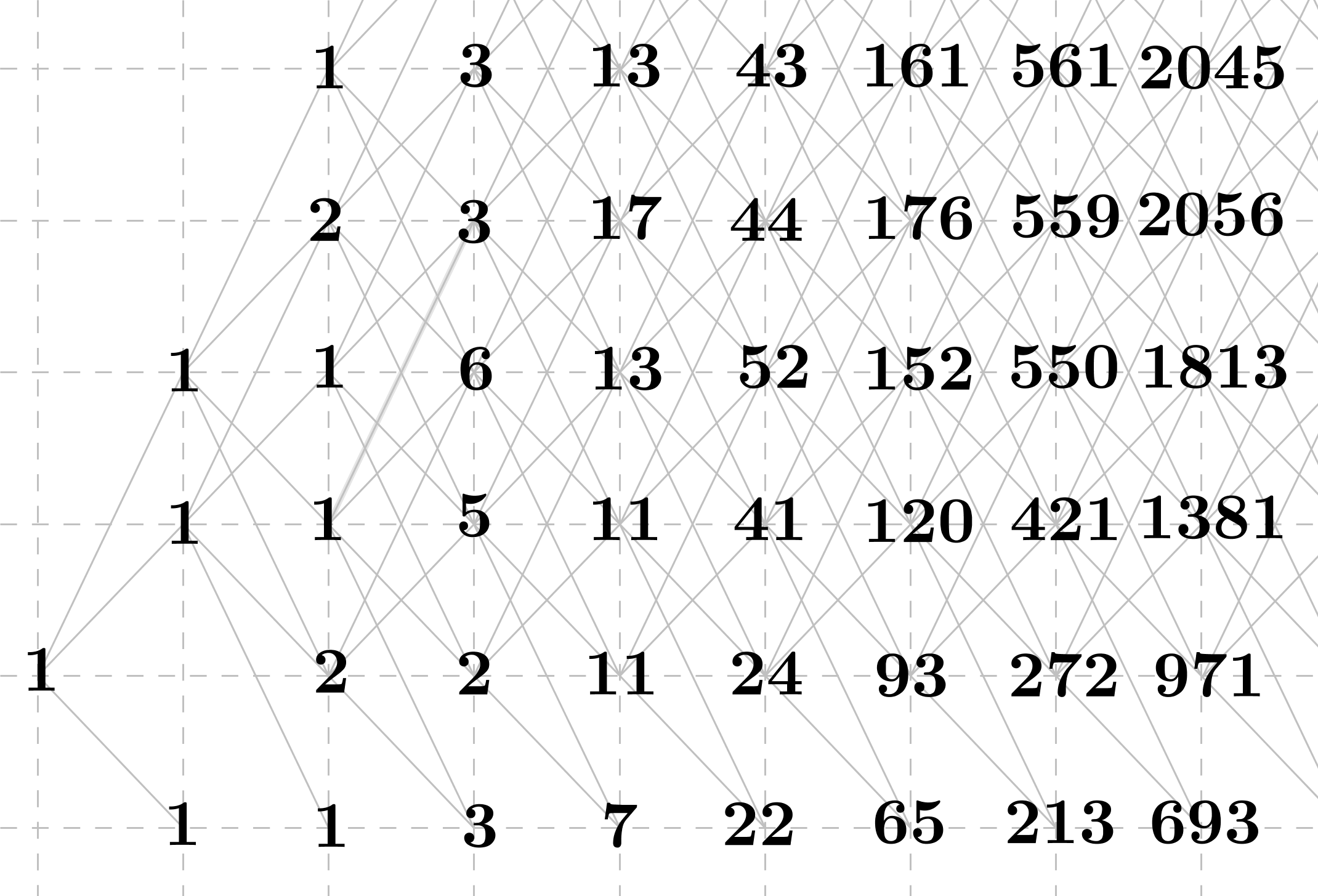}
\caption{Cutting a 4-nomial tree at one unit from its root gives the above picture, which 
thus naturally corresponds to the lattice supporting our lattice
basketball walks. The numbers near each node indicate the number of
walks from the root to this node.\label{4nomialTree}}
\end{figure}

The following proposition gives the exact link between these trees and a generalization of basketball walks.

\begin{proposition}[\sc Link between lattice walks and $m$-nomial trees]
Consider the step sets 
$$	\S_{2n} =\{-n,\dots,-1,1,\dots,n\} \text{ and } 	\S_{2n+1}=\S_{2n}\cup\{0\}.$$
For each step set $S_{m}$, define $T_m(z)$ to be the generating function for walks using steps from $\S_{m}$, starting at the origin and getting absorbed at $-1$.
(By this, we mean that the walks may never touch $y=-1$ except, possibly, at the very last step.)
Then the coefficients of $T_2(z)$ are the Catalan numbers,
the coefficients of $T_3(z)$ are the Motzkin numbers, 
while the coefficients of $T_4(z)$ count our basketball walks from 0 to 1
(walks with steps $\pm 2, \pm 1$, starting at the origin and ending at altitude~1, and never touching 0 in-between). 
\end{proposition}

\begin{proof}While the correspondence is direct for $m\leq 3$, it follows for $m=4$ from a time reversion, 
as each walk from $T_4$ can then be obtained from $G_{0,1}$ and vice versa (see Table~\ref{fig:4types2}). Thus, $T_4(z)=G_{0,1}(z)$.
\end{proof}

\medskip
\begin{table*}[!hb]
		\small
		\begin{center}\renewcommand{\tabcolsep}{3pt}
			\begin{tabular}{|c|c|}
				\hline
				$T_4$ & $G_{0,1}$\\
				\hline
				\begin{tabular}{c} 				
					{\includegraphics[width=71mm]{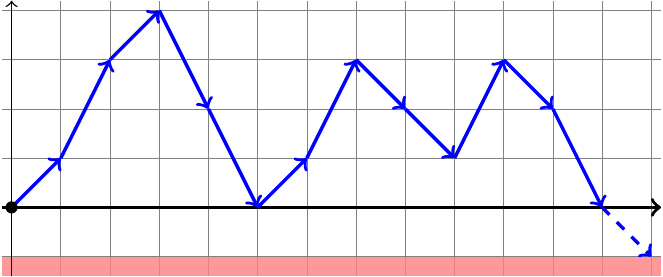}} 
					\\ 
					last step is a 1-step down 
				\end{tabular}
				& \begin{tabular}{c} 				
					{\includegraphics[width=71mm]{G01walk1jump}} 
					\\
					first step is a 1-step up
				\end{tabular} \\
				\hline
				\begin{tabular}{c} 
					\includegraphics[width=71mm]{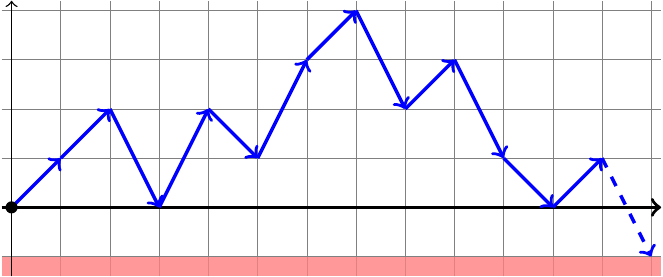} 
					\\ 
					last step is a 2-step down\\ 
				\end{tabular}
				& \begin{tabular}{c} 
					{\includegraphics[width=71mm]{G01walk2jump}} 
					\\ 
					first step is a 2-step up\\ 
				\end{tabular}\\
				\hline
			\end{tabular}
		\end{center}
		\caption{\label{fig:4types2} 
			By time reversal, $T_4(z)=G_{0,1}(z)$.}
	\end{table*}


\subsection{Increasing trees and basketball walks}

A \textit{unary-binary tree} is an ordered tree such that each node has $0,1$, or~$2$ children. 
An \textit{increasing unary-binary tree on $n$ vertices} is a unary-binary tree with $n$ vertices labelled $1,2,\dots,n$ such that the labels along each walk from the root are increasing (cf.~\cite[p.~51]{Stanley86}).
Given an increasing unary-binary tree~$T$, we associate with $T$ the \textit{permutation} $\sigma_T$ constructed by reading the tree left to right, level by level, starting at the root.
A permutation $\sigma$ is said to \textit{contain the pattern $\pi$}
if there exists a subsequence of $\sigma$ that has the same relative
order as $\pi$. Otherwise, $\sigma$ is said to \textit{avoid the
  pattern $\pi$}. For example, the permutation $\sigma=14235$ contains
the pattern $213$ because $\sigma$ contains the subsequence $425$, in
which the numbers have the same relative order as in $213$, while the permutation $12453$ avoids $213$. 

Manda Riehl initiated studies of increasing trees for which the associated permutation avoids a given pattern (see also~\cite{Riehl16}).
By a computer program, she obtained the first terms of the corresponding sequences for patterns of length~3.
She observed that ``the number of increasing unary-binary trees with associated permutation avoiding 213''
seems to coincide with  sequence \oeis{A166135}, which we proved to
count basketballs walks from altitude~0~to altitude~1.
Figure~\ref{trees} shows a verification of this claim for $n=5$: 
there are $39$ increasing unary-binary trees on $5$ vertices, among them, $22$ correspond to permutations avoiding the pattern $213$.
(The forbidden subsequences are highlighted in red. The trees in black all avoid $213$. The trees are grouped according to their associated permutations. Tree labels are read left to right.)
\begin{figure} 
	\vspace{.1cm}
	\begin{center}
		\includegraphics[scale=0.66]{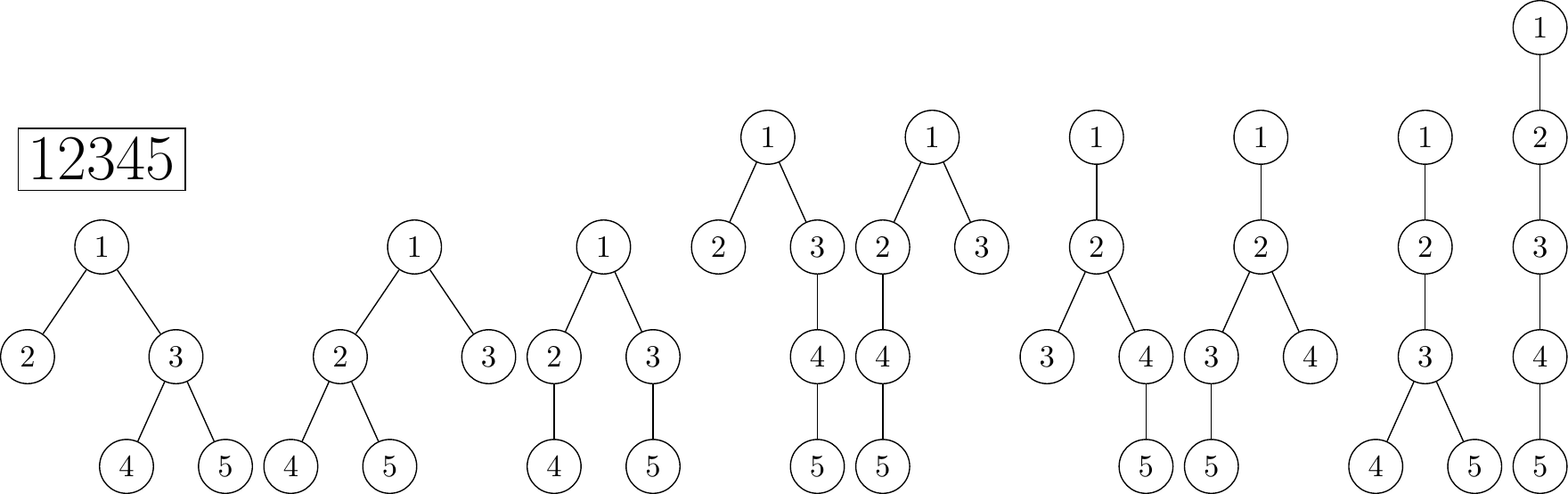}
	\vspace{.2cm}
		\includegraphics[scale=0.66]{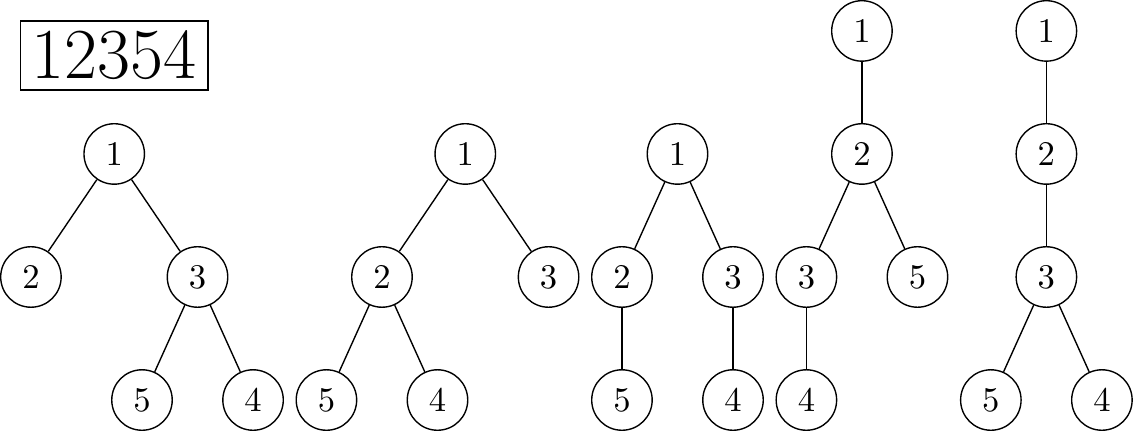}
	\vspace{.2cm}	
		\includegraphics[scale=0.66]{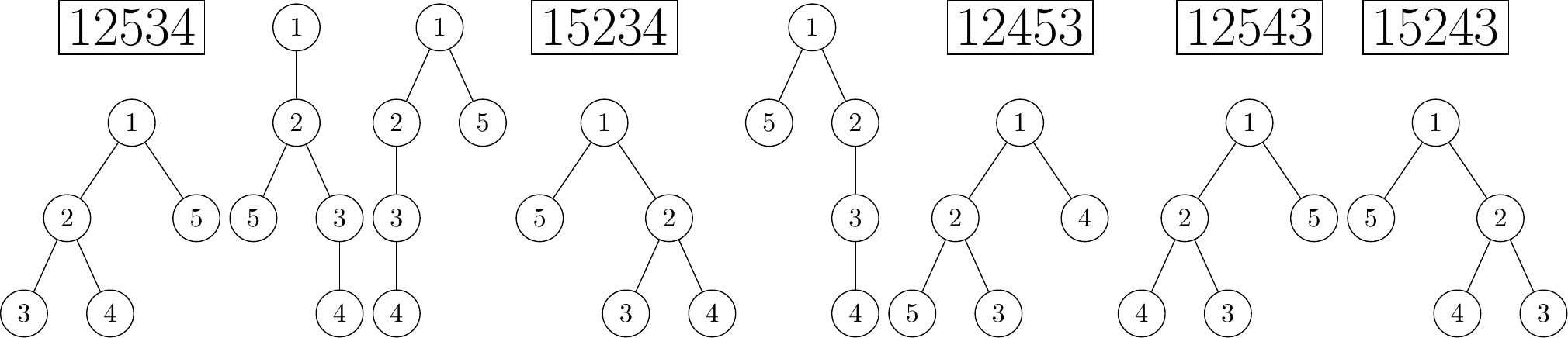}
	\end{center}	
\vspace{-5mm}
\noindent\rule{\textwidth}{1pt}
	\begin{center}
		\includegraphics[scale=0.66]{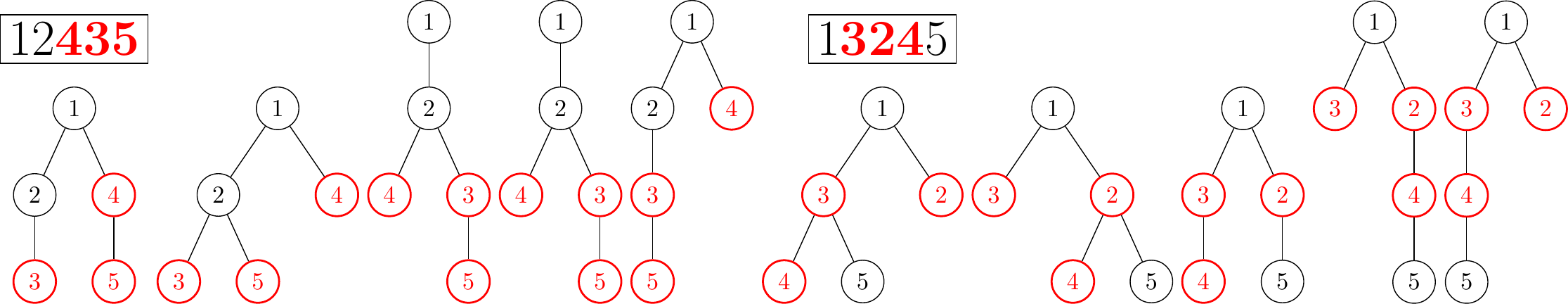}
	\vspace{.2cm}	
		\includegraphics[scale=0.66]{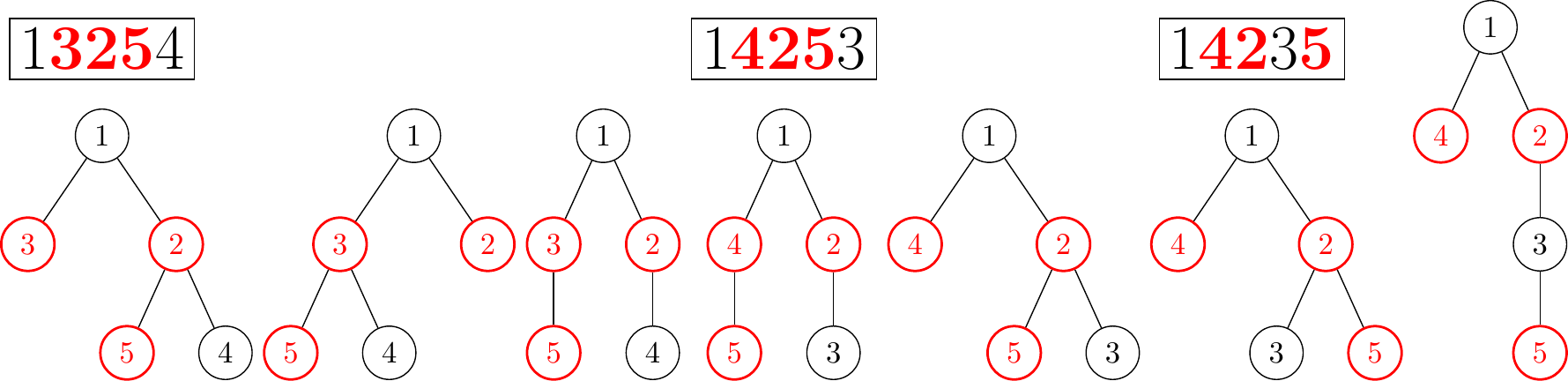}
	\end{center}
	\caption{All increasing unary-binary trees with 5 nodes, where patterns $213$ are marked in red. There are 22 trees (drawn in black) for which the associated permutation avoids this pattern.\label{trees}}
\end{figure}

Here is the reformulation of Riehl's conjecture which takes into account our findings.

\begin{conjecture} \label{conj:Riehl}
	The number of basketball walks of length $n$ starting at the origin and ending at altitude~$1$ that never touch or pass below the $x$-axis equals
 the number of increasing unary-binary trees on $n$ vertices with associated permutation avoiding $213$.
\end{conjecture}

After the first version of this article was circulated via the {\tt arXiv}, Bettinelli, Fusy, Mailler, and Randazzo~\cite{BettinelliFusy16} 
found a nice bijective proof of this conjecture.

\medskip

How strong is the constraint of avoiding the pattern 213?
For this, we need to compute the probability that an increasing
unary-binary tree avoids the pattern 213. 
Due to Conjecture~\ref{conj:Riehl}, proved in~\cite{BettinelliFusy16}, we know the number 
of increasing unary-binary trees which avoid $213$. Hence,
the question is to compute the total number $t_n$ of increasing
unary-binary trees, which can be done via the so-called boxed product. 

The boxed product (written $\boxproduct$) 
is the combinatorial construction corresponding to a labelled product,
in which the minimal label is forced to be in the first component of
this product (see~\cite{FlSe09}). 
This leads the following recursive decomposition  for binary-ternary
increasing trees $\T$:
$$\T= leaf + root  \boxproduct \T +    root  \boxproduct \T \times \T \,,$$
which translates into the following functional equation for the
corresponding exponential generating function:
$$T(z)= z + \int_0^z T(t) dt + \int_0^z T^2(t) dt\,.$$
By solving the associated differential equation $T'(z) = 1+
T(z)+T^2(z)$,
we obtain
$$T(z)= \frac{\sqrt{3}}{2}  \tan\left(\frac{\pi}{6}+\frac{\sqrt{3}}{2}z \right) -\frac{1}{2}\,.$$
The corresponding Taylor expansion is 
$$T(z)=\sum_{n\geq 1} t_n \frac{z^n}{n!} = z +    \frac{z^2}{2!} + 3 \frac{z^3}{3!} +  9 \frac{z^4}{4!}+ 39 \frac{z^5}{5!}+189\frac{z^6}{6!}+  1107\frac{z^7}{7!}+O(z^8) \,.$$
Singularity analysis on the dominant poles of the $\tan$ function implies that 
$$t_n \sim 3 \sqrt{\frac{3}{2\pi}}   \left( \frac{3^{3/2}}{2 e \pi}\right)^n   \sqrt{n} \, n^n\,.$$

In conclusion, increasing unary-binary trees grow like $n^{1/2} A^n
n^n$, while the same trees avoiding the pattern 213 grow like $n^{-3/2}4^n$.
This observation suggests the following natural conjecture.

\begin{conjecture}[\sc A Stanley--Wilf-like conjecture for pattern avoidance in increasing trees]
Let $\T$ be a class of increasing trees of prescribed arity encoded by a power series $\phi$, i.e., one has  $\T'=z \phi(\T)$.
Then the number $a_n$ of such trees avoiding a given pattern satisfies $a_n=O(C^n)$, for some $C$ depending on the pattern and on $\phi$.
\end{conjecture}

This conjecture shares the spirit of the Stanley--Wilf conjecture
(proven by a combination of \cite{KlazAA} and~\cite{MarcusTardos04}), 
which asserted
that any class of pattern-avoiding permutations has an exponential growth rate.  

\subsection{Boolean trees and basketball walks from 0 to 2}
In~\cite{bender91}, Bender and Williamson considered the problem of
bracketing some binary operations (objects that are in bijection with the Boolean trees that we present in Figure~\ref{booleantrees}).
It turns out that this problem is doubly related to our basketball
walks (walks with steps $\pm 1, \pm 2$, always positive). 
This is what we address in the next two propositions.
\begin{figure}[H] 
		\includegraphics[scale=0.88]{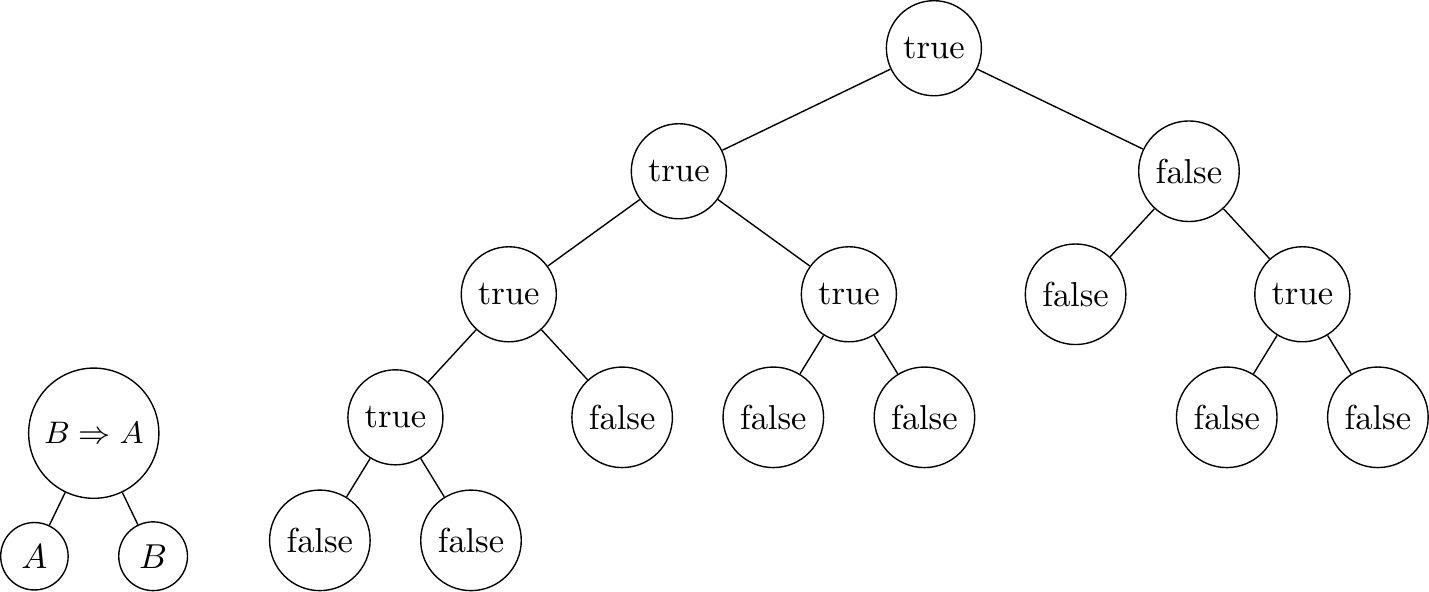}
		\includegraphics[scale=0.88]{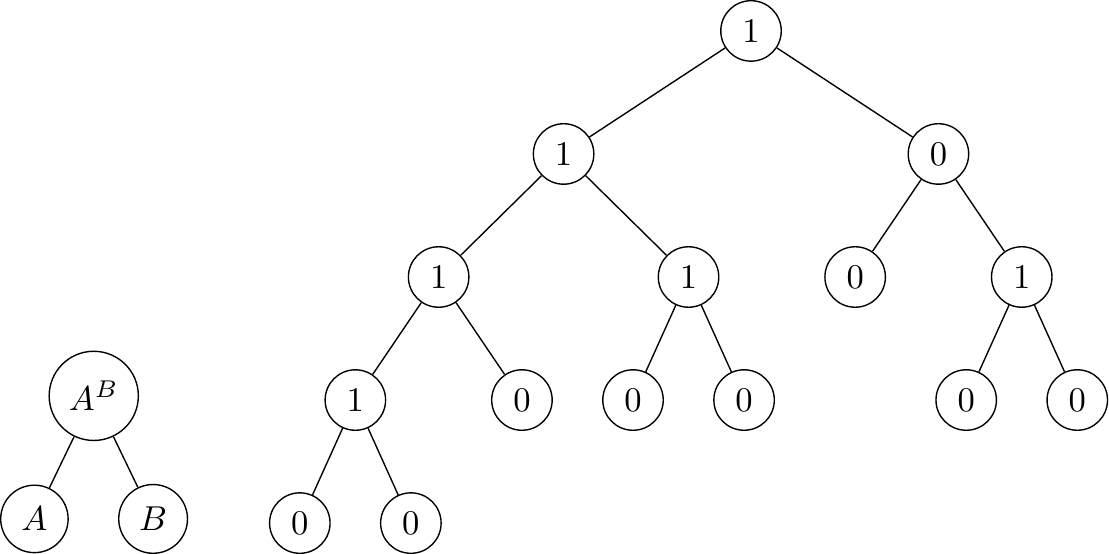}
\caption{Boolean trees (i.e., binary trees where each node is labelled either ``false'' or ``true'') such that a node having children with Boolean value $A$ and $B$ will have the Boolean value ``$B \Rightarrow A$''.
}\label{booleantrees}
\end{figure}

\begin{proposition} 	
Under the conventions $1^1=1^0=0^0=1$ and $0^1=0$,	
the number of bracketings of $n+1$ zeroes $0$\^{}$\cdots$\^{}$0$ giving result $1$
is equal to the number of basketball walks from altitude~$0$ to 
altitude~$2$ of length $n$. 
\end{proposition}
\begin{proof}
Let $W(z)$ {\em(}respectively $Z(z)${\em)} be the generating function 
for the number of bracketings of $n$ zeroes 0\^{}0\^{}$\cdots$\^{}0 
producing result $1$ (respectively $0$).
The objects that are counted by $W(z)$ are of the form
(``1'')\^{}(``1''), (``1'')\^{}(``0''), or (``0'')\^{}(``0''),
where ``1'' stands for a bracketing producing the result~1, and
``0'' stands for a bracketing producing the result~0.
This observation translates into the generating function equation
	\begin{equation}
		W(z)=W^2(z)+Z(z)W(z)+Z^2(z). \label{eq:W}
	\end{equation}
	Similarly, a bracketing producing $0$ may either be a single $0$ or a
bracketing of the form (``0'')\^{}(``1''). This yields the equation
	\begin{equation}
		Z(z)=z+Z(z)W(z). \label{eq:Z}
	\end{equation}
	Let $C(z):=Z(z)+W(z)$. Equations~\eqref{eq:W} and~\eqref{eq:Z} imply 
$	C(z)=1+zC^2(z)$,	i.e.,
	$
		C(z)=\frac{1}{2z}-\frac{1}{2z}\sqrt{1-4z}. 
	$
	This is not a surprise because $W+Z$ corresponds to well parenthesized words, known to be counted by Catalan numbers.

We may ``replace'' $W(z)$ by $C(z)$ in Equation~\eqref{eq:Z}.
This leads to
	\begin{equation}
	Z(z)=z+Z(z)(C(z)-Z(z)).
	\end{equation}
Solving for $Z(z)$, we obtain
	\begin{align}\label{eq:ZZ}
	\begin{aligned}
	Z(z) &=\frac{C(z)-1+\sqrt{(C(z)-1)^2+4z}}{2} \\
	     &=-\frac{1}{4}-\frac{1}{4}\sqrt{1-4z}+\frac{1}{4}\sqrt{2+12z+2\sqrt{1-4z}}.
	\end{aligned}
	\end{align}
	Therefore, we get
	\begin{equation}
	W(z)=C(z)-Z(z)
	=\frac{3}{4}-\frac{1}{4}\sqrt{1-4 z}-\frac{1}{4}\sqrt{2 + 12 z + 2\sqrt{1-4 z}}.
	\end{equation}
		Comparison of this expression with 
Expression~\eqref{eq:G02analytic} for $G_{0,2}(z)$ 
		shows that $W(z)=zG_{0,2}(z)$.
\end{proof}
We leave it to the reader to find a bijective proof between bracketings of
0\^{}\dots\^{}0 having value~1 and basketball walks from altitude~0 to
altitude~2.

\begin{proposition}
	The number of basketball walks of length $n$ starting at the
        origin, ending at altitude~$1$, never running below the $x$-axis in-between,  is equal to the number of bracketings of $n+2$ zeroes $0$\^{}$0$\^{}$\cdots$\^{}$0$ producing result~$0$.
\end{proposition}

\begin{proof}
	The generating function $F_1(z)$ for walks ending at $1$ is given by~\eqref{eq:G2} in the form
	\begin{equation}
		F_1(z) =G_{1,2}(z) = \frac{u_1(z)u_2(z)+u_1(z)+u_2(z)}{z}.
	\end{equation}
	The generating function $Z(z)$ for the number of bracketings of $n$ zeroes 0\^{}$\cdots$\^{}0 having value $0$ is given by \eqref{eq:ZZ}.
	Substitution of the closed-form expressions for the small roots
	into $F_1(z)$ yields $z^2F_1(z)=Z(z)$. 
This establishes the claim.
\end{proof}

\section{Conclusion}

In this article, we show how to derive closed-form expressions for the
enumeration of lattice walks satisfying various constraints (starting
point, ending point, positivity, allowed steps, \dots).
The key is a proper use of the Lagrange--B\"urmann inversion in
combination with
the expressions given by the kernel method. 
This technique admits many extensions, which will work in a similar way: it
is possible to extend it to walks in which we want to keep track of some
parameters (marking a specific step, pattern, altitude, \dots), 
allowing an infinite set of steps, or unbounded steps 
(this would encode what is called
catastrophes in queuing theory language). It is also possible to consider
other constraints, such as to force the walk to live in some cone or
to have some forbidden patterns. In all these cases, the kernel method
will give a closed-form expression for the generating function, in
terms of the roots of the kernel, and thus, our mix of kernel method
and Lagrange--B\"urmann inversion
 will lead in these situations also to some closed-form expression for
 the coefficients of the generating function (in terms of nested sums
 of binomials).  

In several cases, these nested sums of binomials provide the nice
challenge of finding bijective proofs. 
It is satisfying to find {\it some} formula for the enumeration of
certain lattice paths which is efficient (in terms of
algorithmic complexity), 
but the fact that many of these sums involve only positive terms
is an indication that combinatorics has still its word to say on these
formulas. 

The holonomic approach, as well illustrated by the book of Petkov{\v s}ek, Wilf, and Zeilberger~\cite{pewz96}, or Kauers and Paule~\cite{KauersPaule11},
is a way to prove that different binomial expressions correspond in fact to the same sequence. 
It remains an open question to know which methods can lead to the most concise formula:
the platypus algorithms and the Flajolet--Soria formula~\cite{BaFl02,BanderierDrmota15},
or the cycle lemma, and extraction of diagonals of rational functions  
seem to indicate that we could in fact need an arbitrarily large amount of
nested sums. 
In some cases, one can reduce the number of nested sums with
techniques from symbolic summation theory (e.g., by $\Sigma \Pi$
extension theory~\cite{Schneider07},  
or geometric simplifications in diagonal extractions of rational functions~\cite{BostanLairezSalvy15}),
but it is still unknown if, for the directed lattice path models we considered,  there is a miraculous simple formula (with just one or two nested sums). 

\bigskip
{\em Acknowledgments:} We thank the organizers of the 
8th International Conference on Lattice Path Combinatorics \& Applications,
which provided the opportunity for this collaboration.
Sri Gopal Mohanty played an important role in the birth of this
sequence of conferences, and his book~\cite{Mohanty79} was the first
one (together with the book of his Ph.D. advisor
Tadepalli Venkata Narayana~\cite{Narayana79}) to spur strong interest
in lattice path enumeration.    
We are therefore pleased to dedicate our article to him.

\bibliographystyle{plain}
\bibliography{basket}
\end{document}